\documentclass[reqno,10pt]{amsart}
\usepackage{amsmath,amsthm,amsfonts,color,graphicx}
\usepackage[latin1]{inputenc}
\usepackage[makeroom]{cancel}
\usepackage{tikz}
\usetikzlibrary{decorations.pathreplacing}

\newcommand{\bbn}{\mathbb{N}}

\newcommand{\dhr}{\mathrel{\lhook\joinrel\relbar\kern-.8ex\joinrel\lhook\joinrel\rightarrow}}

\newcommand{\seqk}[1]{(#1 _k)_{k\in\bbn}}

\newtheorem{thm}{Theorem}[section]
\newtheorem{lem}[thm]{Lemma}

\newtheorem{prop}[thm]{Proposition}

\newtheorem{rem}[thm]{Remark}

\newtheorem{deff}[thm]{Definition}

\DeclareMathAlphabet{\mathpzc}{OT1}{pzc}{m}{it}

\begin{document}
\bibliographystyle{plain}

\title[Numerics in a Box]{Using an Encompassing Periodic Box to Perform Numerical
  Calculations on General Domains}

\author{Patrick Guidotti}
\address{University of California, Irvine\\
Department of Mathematics\\
340 Rowland Hall\\
Irvine, CA 92697-3875\\ USA }
\email{gpatrick@math.uci.edu}

\begin{abstract}
This paper shows how numerical methods on a regular grid in a box can
be used to generate numerical schemes for problems in general smooth
domains contained in  the box with no need for a domain specific
discretization. The focus is mainly be on spectral discretizations due
to their ability to accurately resolve the interaction of finite order
distributions (generalized functions) and smooth functions. Mimicking
the analytical structure of the relevant (pseudodifferential)
operators leads to viable and accurate numerical representations and
algorithms. An important byproduct of the structural insights gained
in the process is the introduction of smooth kernels (at the discrete
level) to replace classical singular kernels which are typically
used in the (numerical) representations of the solution. The new
kernel representations yield enhanced numerical resolution and, while
they necessarily lead to significantly higher condition numbers, they
also suggest natural and effective ways to precondition the systems.
\end{abstract}

\subjclass[1991]{}

\maketitle

\section{Introduction}
It is the primary goal of this paper to develop a framework which allows one
to extend the benefits of numerical spectral methods from boxes to
arbitrary geometry domains. The idea is to simulate problems on
domains $\Omega\subset B$ located inside a periodicity box $B=[-\pi,\pi]^N$ by
using spectral approximation of generalized functions
(distributions). This is best illustrated in the case of boundary
value problems, which are also an important application of the
method. Consider the boundary value problem
$$\begin{cases}
 \mathcal{A}u=f&\text{in }\Omega,\\
 \mathcal{B}u=g&\text{on }\partial \Omega,
\end{cases}$$ 
for some generic differential operator $\mathcal{A}$ and boundary
operator $\mathcal{B}$. It is no restriction to assume that the
operators and the data be defined everywhere in the box $B$. One
obtains a numerical approximation for the boundary value problem in
the following manner. First discretize the periodicity box $B$ by a
regular grid $G^m=\{x^m_j\,:\, j\in \mathbb{Z}^m_N\}$ with $2^{mN}$
points if working in dimension $n\in \mathbb{N}$, then, independently,
discretize the boundary $\Gamma$ of the domain $\Omega$ by a subset
$\Gamma^n=\{y_1,\dots,y_n\}\subset \Gamma$. Choose a discretization $A^m$
of the operator $\mathcal{A}$ which operates on the grid $G^m$ and
find a solution $v^m:G^m\to \mathbb{R}$ of
$$
 A^mv^{m}=f^m
$$
in $G^m$ for a discretization of $f$. With that in hand, generate
numerical approximations for functions $\psi ^m_k:G^m\to \mathbb{R}$,
$k=1,\dots,n$, in the kernel of the operator $\mathcal{A}_\Omega$ on
$\Omega$ and try to adjust the solution $v^m$ by a linear combination 
$$
 w^{m,n}=\sum _{k=1}^nw^{m,n}_k\psi ^m_k
$$
of these kernel elements in order for $u^{m,n}=v^m+w^{m,n}$ to
satisfy a discretization $\mathcal{B}^{m,n}u^{m,n}=g^n$ of the
boundary condition for a discretization $g^n:\Gamma^n\to \mathbb{R}$
of $g$. Choosing $\mathcal{B}$ to be the trace operator
$\gamma_\Gamma$ at first for ease of presentation, this can be done as
follows. Approximate $\delta_{y_k}$ for $k=1,\dots,n$ by its spectral
representation $\delta^m_{y_k}$ on the periodic grid $G^m$ and insist
that
$$
 \langle \delta^m_{y_j},u^{m,n}\rangle _{q^m}=g^n_j,\: j=1,\dots, n,
$$
where $\langle\cdot,\cdot \rangle _{q^m}$ is the discrete duality
pairing (scalar product) discretizing the continuous duality pairing
$\langle\cdot,\cdot \rangle_{\mathcal{D}_\pi ',\mathcal{D}_\pi}$
between periodic distributions and test functions. Details will be
given in the rest of the paper. Following the strategy outlined above
leads to a system for the unknown $w^{m,n}$ of the form
$$
 \langle \delta^m_{y_j},w^{m,n}\rangle _{q^m}=\sum _{k=1}^n\langle
 \delta^m_{y_j},\psi^m_k\rangle
 _{q^m}w^{m,n}_k=\sum_{k=1}^nM_{jk}w^{m,n}_k=g^n_j-\langle  
 \delta^m_{y_j},v^{m}\rangle _{q^m},\: j=1,\dots, n, 
$$
One can think of $\delta ^m_{y_\cdot}$ as the discrete kernel of the
trace operator $\gamma _\Gamma$. It is therefore possible to deal with
a more general boundary operator $\mathcal{B}$ by deriving a
``natural'' numerical approximation $B^m_{y_j}$ of its
distributional kernel for $j=1,\dots,n$. This would lead to the system
$$
 \langle B^m_{y_j},w^{m,n}\rangle _{q^m}=\sum _{k=1}^n\langle
 B^m_{y_j},\psi^m_k\rangle _{q^m}w^{m,n}_k=g^n_j-\langle 
 B^m_{y_j},v^{m}\rangle _{q^m},\: j=1,\dots, n, 
$$
In order to obtain a numerical method it remains to generate the
kernel functions $\psi _k^m$ for $k=1,\dots,m$. This can be done in
many different ways. In order to, at first, make a connection explicit
to pseudo-differential operators, again consider
$\mathcal{B}=\gamma_\Gamma$ and proceed in the following
manner. Take the spectral approximation $\delta^m_{y_k}$ for
$k=1,\dots,n$ and set 
$$
 \psi_k^m=(A^m)^{-1}\delta^m_{y_k}.
$$
Since the Dirac distribution is ``supported'' on the singleton $\{y_k\}$, the
function $\psi _k^m$ will indeed ``lie'' in the kernel of $A^m$ over
$\Omega$. Since
these functions are ``peaked'' at different locations $y_k$, they will
be linearly independent. The matrix $M$ in the system for the unknown
$w^{m,n}$ is therefore given by
$$
 M_{jk}=\langle\delta ^m_{y_j},(A^m)^{-1}\delta ^m_{y_k}\rangle
 _{q^m},\: j,k=1,\dots,n. 
$$
Latter can be recognized as the discrete counterpart of 
$$
 m(y,\eta)=\langle \delta _{y},\mathcal{A}^{-1}\delta
 _{\eta}\rangle,\: y,\eta\in \Gamma,
$$
the distributional kernel of a pseudodifferential operator
$\mathcal{A}^{-1}$ on the
boundary curve $\Gamma$. This connection is made more precise in the
rest of the paper and provides a framework in which to obtain
analytical proofs for the numerical methods introduced. For
implementation purposes, however, it is best to proceed in a somewhat
different way when constructing the 
kernel functions $\psi_k^m$. Instead of using the ``rougher'' Dirac
distributions used above, it is better to replace them by smooth
functions $\varphi _{\tilde y_k}$ which are supported outside of
$\Omega$ with support ``centered'' at $\tilde y_k=y_k+\delta
\nu_\Gamma(y_k)$ for $\delta>0$ where $\nu_\Gamma(y_k)$ is the unit
outer normal to the boundary $\Gamma$ at the point $y_k$. After
discretization this leads to the alternative matrix
$$
 \widetilde{M}_{jk}=\langle\delta ^m_{y_j},(A^m)^{-1}\varphi ^m_{\tilde y_k}\rangle
 _{q^m},\: j,k=1,\dots,n,
$$
which is the discretization of a smoothing operator with kernel
$$
 \widetilde{m}(y,\eta)=\langle \delta _{y},\mathcal{A}^{-1}\varphi
 _{\tilde\eta}\rangle,\: y,\eta\in \Gamma.
$$
As such, $M$ will be easier to capture numerically (fast convergent
expansion of its kernel function) but also badly conditioned (as a
smoothing and thus compact operator with unbounded inverse -- read less
diagonally dominant). In spite
of this, ``natural'' and effective preconditioning procedures can be
devised which completely remove this drawback. 

The method above can be thought of as a fully discrete boundary
integral method. As such it does not rely on the availability of an
explicit analytic representation of the kernels involved and is
therefore applicable to non-constant coefficient differential
operators. Even for situations where analytic representations of the
kernel are known, the singularity shifting/removal procedure employed
above offers an effective and accurate numerical discretization method
which completely avoids the need to find ways to numerically deal with
the singularity of the kernel function.

One of the main reason for developing the method is its applicability
to the numerical computation of solutions to moving boundary value
problems. The fact that the domain evolves in time would in general require  
continuous remeshing of the varying computational domain. In the 
approach presented here, the encompassing computational domain remains
unchanged during the evolution and computation of the moving boundary is
reduced to tracking the location of its discretization points.

The rest of the paper
is organized as follows. In the next section some preliminary results
are obtained which highlight the main features of the underlying
spectral approach to approximating generalized functions and
test functions in the context of periodicity and for boundary value
problems, in particular; in Section 3 details are given about the
discretizations used in the concrete examples studied in Section
4. The following section is dedicated to the specifics of the
numerical implementation and to numerical experiments which illustrate
the main theoretical insights and the advantages of the proposed
method. A brief conclusion ends the paper.
\section{Preliminaries}
\subsection{Setup}
Before working with the relevant discretizations, the stage is set by fixing the
analytical context which will very much guide the numerical procedures developed in
the rest of the paper. Let $B=[-\pi,\pi)^N$ be the periodicity box bounding the area
of interest. Extensive use will be made of distributions and of test functions. Latter
are periodic smooth functions belonging to one of the following useful spaces
\begin{align}
 \mathcal{D}_\pi(B)&=\mathcal{D}_\pi=\big\{ \varphi\in
 \operatorname{C}^\infty(\mathbb{R}^N)\, \big |\, \varphi\text{ is
  }2\pi\text{-periodic}\big\} \\
 \mathcal{D}^m_\pi(B)&=\mathcal{D}^m_\pi=\big\{ \varphi\in
 \operatorname{C}^m(\mathbb{R}^N)\, \big |\, \varphi\text{ is
  }2\pi\text{-periodic}\big\} \\
 \mathcal{D}_0(B)&=\mathcal{D}_0=\big\{ \varphi\in
 \operatorname{C}^\infty(\mathbb{R}^N)\, \big |\,
  \operatorname{supp}(\varphi)\subset\subset B\big\}
\end{align}
where $m\in \mathbb{N}$. The first space carries its standard locally
convex topology generated by the family of seminorms $\{p_m:m\in
\mathbb{N}\}$ given by
$$
 p_m(\varphi)=\sup _{|\alpha|\leq m}\| \partial ^\alpha \varphi\| _\infty,
$$
the second is a Banach with respect to the norm $p_m$, and the last carries the
natural inductive limit-Fr\'echet topology, i.e. the coarsest topology
which makes the inclusions 
$$
 \mathcal{D}_K(B)=\big\{ \varphi\in
 \operatorname{C}^\infty(\mathbb{R}^N)\, \big |\,
  \operatorname{supp}(\varphi)\subset K\big\}\hookrightarrow \mathcal{D}_0(B)
$$
continuous for all $K=\overline{K}\subset\subset B$. Notice that $\mathcal{D}_K(B)$
is endowed with the locally convex topology induced by the seminorms
$$
 p_{m,K}(\cdot)=\sup _{|\alpha|\leq m}\| \partial ^\alpha \cdot\| _{\infty,K},
$$
where the additional subscript indicates that the supremum norm is
taken over the set $K$. The space
$$
 \mathcal{D}_\pi '(B)=\mathcal{D}'_\pi=\big\{ u:\mathcal{D}_\pi\to \mathbb{K}\, \big
 |\, u\text{ is linear and continous}\big\}
$$
is then the space of $\mathbb{K}(=\mathbb{R},\mathbb{C})$-valued distributions dual to
$\mathcal{D}_\pi$. On
$\operatorname{L}^2_\pi=\operatorname{L}^2_\pi(B)=\operatorname{L}^2(B)$ there is a 
natural orthonormal basis $(e_k)_{k\in \mathbb{Z}^N}$ given by
$$
 e_k(x)=\frac{1}{(2\pi)^{N/2}}e^{ik\cdot x},\: x\in B,\: k\in \mathbb{Z}^N,
$$
consisting of eigenfunctions of the periodic Laplacian $-\Delta _\pi$. It is
well-known that
$$
 \mathcal{F}: \operatorname{L}^2_\pi\to l^2(\mathbb{Z}^n),\: \varphi=\sum_{k\in
   \mathbb{Z}^n}\hat\varphi _ke_k\mapsto(\hat\varphi)_{k\in \mathbb{Z}^N}
$$
is an isometric isomorphism where
$$
 \hat\varphi _k=\int _B\varphi(x)\bar{e}_k(x)\, dx=\langle\varphi,\bar
 e_k\rangle=(\varphi |e_k). 
$$
In particular one has that $\| \varphi \|
_{\operatorname{L}^2_\pi}=\|\seqk{\hat\varphi}\| _{l^2(\mathbb{Z}^n)}$ and Parseval's identity
$$
 (\varphi |\psi)=\int _B\varphi\,\bar\psi\, dx=\sum _{k\in
   \mathbb{Z}^n}\hat\varphi _k\overline{\hat\psi _k}=(\hat \varphi
 |\hat\psi)\text{ for }\varphi,\psi\in \operatorname{L}^2_\pi.
$$
Notice that the formul{\ae} above use the notations $\langle\cdot,\cdot\rangle$ and
$(\cdot |\cdot)$ for the duality pairing and the scalar product,
respectively. The former is clearly motivated by the natural duality
pairing between distributions and test functions
$$
 \langle\cdot,\cdot\rangle:\mathcal{D}'_\pi\times \mathcal{D}_\pi\to
 \mathbb{K},\:(u,\varphi)\mapsto \langle u,\varphi\rangle=u(\varphi).
$$
Observe that, if $\varphi\in \mathcal{D}_\pi$, then $\partial^\alpha\varphi\in
\operatorname{L}^2_\pi$ for all $\alpha\in \mathbb{N}^n$ and thus
$$
 \partial^\alpha\sum_{k\in
   \mathbb{Z}^N}\hat\varphi_ke_k=\partial^\alpha\varphi=\sum_{k\in
   \mathbb{Z}^N}\widehat{(\partial^\alpha\varphi)}_ke_k=\sum _{k\in 
   \mathbb{Z}^N}(ik)^\alpha\hat\varphi _ke_k,
$$
with convergence in $\operatorname{L}^2_\pi$, owing to well-known properties of the
Fourier transform. Introducing the periodic Bessel potential spaces via
$$
 \operatorname{H}^s_\pi=\operatorname{H}^s_\pi(B)=\big\{ u\in
 \mathcal{D}'_\pi\, \big |\, \sum _{k\in \mathbb{Z}^n}(1+|k|^2)^s\hat u
 ^2_k<\infty\big\}, 
$$
for $s\in \mathbb{R}$ and $\hat u_k=\langle u,\bar e_k\rangle=(u
|e_k)$, it follows easily that
$$
 \sum _{|k|\leq M}\hat\varphi _ke_k\to\varphi\text{ as }M\to\infty,
$$
in $\operatorname{H} ^s_\pi$ for any $s\geq 0$ if $\varphi\in \mathcal{D}_\pi$. By
the well-known embedding
\begin{equation}\label{embed}
 \operatorname{H}^s_\pi\hookrightarrow
 \mathcal{D} ^m_\pi=\mathcal{D}^m_\pi(B), 
\end{equation}
valid for $s>N/2+m$, it then follows that the convergence of the Fourier series
actually takes place in the topology of $\mathcal{D}_\pi$. An important
consequence of this fact is the validity of the following generalized
Parseval's identity 
\begin{gather}\label{parseval}
 \langle u,\varphi\rangle=\langle u,\sum_{k\in \mathbb{Z}^N}\hat\varphi
 _ke_k\rangle=\sum_{k\in \mathbb{Z}^N} \underset{=:\tilde u_k}{\underbrace{\langle
     u,e_k\rangle}}\,\hat\varphi _k=\sum_{k\in\mathbb{Z}^N} \tilde
 u_k\hat\varphi _k,\\
 (u |\varphi )=(u |\sum_{k\in \mathbb{Z}^N}\hat \varphi
 _ke_k)=\sum_{k\in \mathbb{Z}^N}(u|e_k)\hat \varphi_k=(\hat u|\hat
 \varphi)
\end{gather}
for $u\in \mathcal{D}'_\pi$, $\varphi\in \mathcal{D}_\pi$, and
$(u|e_k)=\langle u,\bar{e}_{k}\rangle=\langle u,e_{-k}\rangle$. A distribution $u\in
\mathcal{D}'_\pi$ is said to be of finite order $m\in \mathbb{N}$ if it admits an
estimate of the form
$$
 |\langle u,\varphi\rangle|\leq c\, p_m(\varphi),\: \varphi\in \mathcal{D}^m_\pi,
$$
for a non-negative constant $c$ but not for $m$ replaced by $m-1$. It follows from a
density argument combined with the embedding \eqref{embed} that any finite order
distribution belongs to $\operatorname{H}^{-s}_\pi$ for some finite $s\geq 0$. The
upshot of this is that \eqref{parseval} can be used to evaluate the action of a
finite order distribution on a test function by a fast converging series since $(\tilde
u_k)_{k\in \mathbb{Z}^N}$ is polynomially bounded and
$(\hat\varphi_k)_{k\in \mathbb{Z}^N}$ decays faster than the reciprocal of 
any polynomial in $k$. This, combined with the choice of appropriate discretizations,
will be exploited later to derive highly accurate representations of various
operators (not even necessarily supported on the discretization grid itself). Indeed
many useful basic operations such as differentiation, integration,
evaluation/interpolation are distributions of finite order. It also turns out that,
for many interesting distributions $u\in \mathcal{D}'_\pi$, it will be possible to
compute their Fourier coefficients either exactly or in a highly accurate manner.
\subsection{Simple Illustrative Examples}
Consider first $\delta _{x_0}$ for $x_0\in B$ which is a zero order
distribution. Later $\delta _{x_0}$ will be discretized on a regular grid but $x_0$
will be allowed to be any point in the domain $B$. Then it holds that
$$
 \delta _{x_0}=\sum_{k\in \mathbb{Z}^N}\langle \delta
 _{x_0},e_k\rangle\bar e_k=\sum_{k\in \mathbb{Z}^N}(\delta
 _{x_0}|e_k)e_k. 
$$
Indeed
$$
 \langle\delta _{x_0},\varphi\rangle=\langle\delta _{x_0},\sum _{k\in
   \mathbb{Z}^N}\hat\varphi _ke_k\rangle=\sum_{k\in \mathbb{Z}^N}\langle\delta
 _{x_0},e_k\rangle\langle\varphi,\bar e_k\rangle=\langle\sum_{k\in
   \mathbb{Z}^N}{(\widetilde{\delta _{x_0}})}_k\bar e_k,\varphi\rangle,
$$
where ${(\widetilde{\delta _{x_0}})}_k=e_k(x_0)$. The convergence of this evaluation
series
$$
 \langle\delta _{x_0},\varphi\rangle=\sum_{k\in \mathbb{Z}^N}e_k(x_0)\hat\varphi _k,
$$
is clearly very fast and its coefficients are known either exactly or to a high
degree of accuracy. This seemingly very simple observation will play a crucial role
in the derivation of a highly accurate representation of higher dimensional kernels
related to boundary value problems.
\begin{rem}
While in this paper it will be enough to deal with the evaluation of smooth
functions, such as test functions, in \cite{G062} modifications are presented (in a
non-periodic context) which make it possible to retain good convergence properties
also for piecewise smooth functions, another important class of functions in applications.
\end{rem}
The next example shows how the considerations of Subsection 2.1 provide an abstract
framework in which to understand spectral methods (after discretization). Let
$\varphi\in \mathcal{D}_\pi$ and consider computing
$$
 \partial ^\alpha\varphi(x_0)=\langle(-1)^{|\alpha|}\partial^\alpha\delta
 _{x_0},\varphi\rangle 
$$
at a point $x_0\in B$. In this case
\begin{align*}
 \partial^\alpha\varphi(x_0)&=\langle(-1) ^{|\alpha|}\partial^\alpha \delta
 _{x_0},\varphi\rangle=\sum_{k\in \mathbb{Z}^N} (-1)^{|\alpha|}(\widetilde{\partial^\alpha\delta
 _{x_0}})_k\hat\varphi _k\\&=\sum_{k\in \mathbb{Z}^N}\partial^\alpha e_k(x_0)\hat\varphi _k=
 \sum_{k\in \mathbb{Z}^N}[(ik)^\alpha\hat\varphi_k]e_k(x_0).
\end{align*}
It is again clear that the main advantages lie in the fact that $\varphi$ is smooth
and that the Fourier coefficients of $\partial^\alpha\delta _{x_0}$ are known
exactly.

The next is an example of integration. Take $x_0<x_1$ in the one dimensional box
$B$. Of interest is the computation (eventually numerically) the integral 
$$
 I(\varphi)=\int _{x_0}^{x_1}\varphi(x)\, dx
$$
between the end points given (which are not necessarily on a numerical
grid). Since $I\in \mathcal{D}'_\pi$ is a zero order distribution, it
is possible compute its Fourier coefficients
$$
 \tilde{I}_k=\langle I,e_k\rangle=\int _{x_0}^{x_1}e_k(x)\,
 dx=\frac{1}{(2\pi)^{1/2}}\frac{1}{ik}[e^{ikx_1}-e^{ikx_0}],
$$
again obtaining an explicit formula. Then
$$
 \int _{x_0}^{x_1}\varphi(x)\, dx=\sum_{k\in \mathbb{Z}}\tilde{I}_k\hat\varphi _k,
$$
will provide a fast converging series representation.
\subsection{A simple Boundary Value Problem}
This section concludes with a simple example that  will make the advantages and basic
principle of this approach apparent. They will reappear in the higher dimensional
context with the appropriate adjustments. Consider the following two point boundary
value problem
\begin{equation}\label{1dbvp}\begin{cases}
 -\partial _{xx}u=f&\text{on }(x_0,x_1)\subset B,\\
 u(x_j)=u_j&\text{for }j=0,1.
\end{cases}\end{equation}
Notice that it will be considered as a problem embedded in the periodicity box $B$
(which will later be discretized by a regular grid). Assume that $f\in
\mathcal{D}_\pi$ be given along with $u_j\in \mathbb{R}$ for $j=0,1$. Choose a
function $\psi\in \mathcal{D}_\pi$ satisfying
$$
 \hat\psi _0=1,\: 1-\psi\in \mathcal{D}_0(B),\: \operatorname{supp}(\psi)\subset
 [x_0,x_1]^\mathsf{c},
$$
and define
$$
 P_\psi f=f-\hat f_0\psi
$$
for the datum $f$ and accordingly for any distribution in $\mathcal{D}'_\pi$. This
way, a function $P_\psi f$ is obtained with vanishing average which coincides with
$f$ on $(x_0,x_1)$. When applied to a general distribution $u$ that is
compactly supported inside the box, this operation produces a modified
distribution $P_\psi u$ which coincides with the original $u$ on its
$\operatorname{supp}(u)$ if, without loss of generality, it is assumed
that $\operatorname{supp}(u)\cap
\operatorname{supp}(\psi)=\emptyset$. Next define the operator $G_\pi$
acting on $f$ via 
$$
 \widehat{G_\pi(f)}_k=\begin{cases} 0,&k=0,\\\frac{\hat f_k}{k^2},&k\neq 0
 \end{cases}
$$
so that
$$
 -\partial _{xx}G_\pi(f)=G_\pi(-\partial _{xx}f)=f-P_0(f),
$$
for $P_0f=\hat f_0e_0$, the orthogonal projection onto average
free functions. A solution of \eqref{1dbvp} can be looked for in the
form 
$$
 u=G_\pi \bigl(P_\psi f \bigr)+v,
$$
where $v$ satisfies
$$
 \partial _{xx}v=0\text{ and } v(x_j)=u_j-G_\pi \bigl(P_\psi f
 \bigr)(x_j),\: j=0,1. 
$$
All that remains is to find two linearly independent elements
$v_0,v_1$ in the kernel of $-\partial _{xx}$ on $(x_0,x_1)$ and look
for $v$ in the form $v=\alpha _0v_0+\alpha _1v_1$. In order for the
boundary conditions to be satisfied, one needs that 
$$
 \alpha _0v_0(x_j)+\alpha_1v_1(x_j)=u_j-G_\pi \bigl(P_\psi f
 \bigr)(x_j)=:\beta _j,\: j=0,1. 
$$
By choosing $v_k=G_\pi(P_\psi\delta _{x_k})$ for $k=0,1$, this leads to
the matrix $M=[v_k(x_j)]_{j,k=0,1}$ with 
\begin{equation}\label{thematrix1d}
 M_{jk}=\langle\delta _{x_j},G_\pi(P_\psi\delta _{x_k})\rangle,
\end{equation}
which is a kind of Green's function ``$M=M(x_j,x_k)$'', $j,k=0,1$. The
crucial observation is that all ingredients $\delta
_{x_j},G_\pi,P_\psi\delta _{x_k}$ allow for spectral representations
in the periodicity interval $B$ (regardless of whether $x_j$ for
$j=0,1$ are or are not grid points after discretization). The
convergence is, however, limited  
by the fact that, in \eqref{thematrix1d}, $\delta_{x_j}$ is of zero
order and that $G_\pi(P_\psi\delta _{x_k})$
is of limited smoothness (slightly better than
$\operatorname{H}^1_\pi$). This, however, can be alleviated by
replacing $P_\psi\delta _{x_k}$ by either $P_\psi\delta _{\tilde x_k}$
with 
$$
 \tilde x_k=x_k+\delta\nu(x_k)\text{ for }\nu(x_k)=(-1)^{k+1}\text{
   and }k=0,1, 
$$
or, even better, by $P_\psi\varphi _{\tilde x_k}$ for a smooth test function
$\varphi _{\tilde x_k}\in \mathcal{D}_0$ supported in a neighborhood $U_{k}$ of
$\tilde x_k$ with $U_k\cap(x_0,x_1)=\emptyset$ and $U_k\cap
\operatorname{supp}(\psi)=\emptyset$ to obtain
$$
 \widetilde M_{jk}=\langle\delta _{x_j},G_\pi\bigl(P_\psi\varphi_{\tilde
   x_k}\bigr)\rangle. 
$$
It is easily checked that $M$ is invertible for $x_0\neq x_1$. Taking
$\tilde x_k$ not too far from $x_k$ and $\varphi _{\tilde x_k}\simeq\delta _{\tilde
  x_k}$, it follows that $\widetilde M\simeq M$ is also invertible. The upshot is,
clearly, that $\widetilde M$ allows for a fast converging
representation of its entries. To conclude this simple example one has that
$$
 u=G_\pi(P_\psi f)+[v_0\: v_1]\widetilde M^{-1}\beta,
$$
for $\beta=[\beta _0\: \beta _1]^\top$.
\begin{rem}
It is to be observed that, after discretization, all basic ingredients $\delta
_{x_j},G_\pi,P_\psi\varphi _{\tilde x_k},\psi$ will have highly accurate grid
representations, even if $x_j,\tilde x_k$ do not lie on the
grid. Owing either to the availability of exact Fourier coefficients
or to their smoothness, the additional discretization
error incurred when going to a finite dimensional representation is as
small as can be hoped for. 
\end{rem}
\section{Discretization}
\subsection{One Dimension}
In order to rip the benefits of the above considerations the interval
$B_1=[-\pi,\pi)$ is discretized at $m\in\mathbb{N}$ (even) equidistant
points $(x^m_j)_{j=0,\dots,m-1}$ where
$$
 x^m_j=-\pi+\frac{2\pi}{m}j,\: j=0,\dots,m-1.
$$
This will be sometimes referred to as the grid $G^m_1$ of size $m$ in
dimension $n=1$. As pointed out in \cite{G062}, the choice of grid has to
be complemented by an appropriate choice of corresponding quadrature
rule $q^m=(q^m_j)_{j=0,\dots,m-1}$ such that
$$
 \langle\mathbf{1}^m,\varphi ^m\rangle _{q^m}=\mathbf{1}^m\cdot
 _{q^m}\varphi^m=q^m\cdot\varphi ^m=\sum _{j=0}^{m-1}\varphi
 ^m_jq^m_j\to\int _{B_1}\varphi(x)\, dx\text{ as
 }m\to\infty,\:\varphi\in \mathcal{D}_\pi,
$$
for the constant function $\mathbf{1}$ with value $1$ and for
$$
 \varphi
 ^m=P_\mathcal{P}(\varphi)=\bigl(\varphi(x^m_j)\bigr)_{j=0,\dots,m-1},
$$
the (physical space) projection of the test function $\varphi$ on the
grid. It is also required that the quadrature rule satisfy
$$
 e^m_j\cdot _{q^m}\bar e^m_k=\delta _{jk},\: -m/2\leq j,k\leq m/2-1,
$$
for the basis vectors $e_j$, $j=-m/2,\dots,m/2-1$, where again the
superscript indicates projection (by evalutation) on the grid.
\begin{deff}
A discretization pair $(x^m,q^m)$ on $B_1$ satsfying the above
properties is called {\em  faithful discretization}.
\end{deff}
The trapezoidal rule, for which it holds that
$$
 q^m=\frac{2\pi}{m}(1,\dots,1),
$$
has this property of preserving the duality pairing and the
orthogonal structure of the continuous setting. Many basic, useful
distributions, such as $\delta _{x_0}$ for any $x_0\in[-\pi,\pi)$,
cannot be directly evaluated at points (short of obtaining a vanishing
projection for all non-grid points $x_0$). It is then better to use an
approximation based on Fourier coefficients and given by
$$
 u^m=P_\mathcal{F}(u)=\sum _{k=-m/2}^{m/2-1}\tilde u_k\bar e^m_k=\sum
 _{k=-m/2}^{m/2-1}\tilde u_k P_\mathcal{P}(\bar e_k),\: u\in\mathcal{D}_\pi.
$$
The reason for this is that, in practice, one often has analytical
knowledge of the coefficients $\tilde u_k$ or the ability to compute
them to a high degree of accuracy.
\begin{rem}\label{delta exact}
Observing that $\delta ^m_{x_0}=\sum _{k=-m/2}^{m/2-1}e_k(x_0)\bar
e_k^m$ and assuming that $x_0=x^m_{j_0}$ is one of the grid points,
one has that
\begin{align*}
 \delta ^m_{x^m_{j_0}}(x^m_j)&=\sum _{k=-m/2}^{m/2-1}e_k(x^m_{j_0})\bar
 e_k(x^m_j)=\sum _{k=-m/2}^{m/2-1}e^{i(j_0-k)\pi}e_{j_0}(x^m_k)\bar
 e_j(x^m_k)e^{i(k-j)\pi}\\ &=e^{i(j_0-j)\pi}\sum _{k=-m/2}^{m/2-1}
 e_{j_0}(x^m_k)\bar e_j(x^m_k)=\frac{m}{2\pi}e^{i(j_0-j)\pi}e_{j_0}^m\cdot
  _{q^m}\bar e_j^m\\&=\frac{m}{2\pi}\delta _{jj_0} 
\end{align*}
since
\begin{equation*}
 e_k(x^m_j)=\frac{1}{\sqrt{2\pi}}e^{-ik\pi+ik\frac{2\pi}{m}j}
 =\frac{1}{\sqrt{2\pi}}e^{i(j-k)\pi}e^{ij(-\pi+\frac{2\pi}{m}k)}=e^{i(j-k)\pi}e_j(x^m_k).
\end{equation*} 
It is seen that $P_\mathcal{F}(\delta _{x_0})$ evaluates exactly (to
the discrete Dirac function) if $x_0$ is a grid point, while, for
$x_0\in[-\pi,\pi)\setminus G^m_1$, it has oscillatory character. In
any case one has that
$$
 \langle\delta ^m_{x_0},\varphi ^m\rangle_{q^m}\to\langle\delta
 _{x_0},\varphi\rangle=\varphi(x_0)\text{ as }m\to\infty,
$$
for any $\varphi\in\mathcal{D}_\pi$, with fast convergence.
\end{rem}
\begin{rem}
The alternating point trapezoidal rule of quadrature given by
$q^m=\frac{2\pi}{m} (2,0,\dots,2,0)$ can also be used instead of the regular
trapezoidal rule as it has been observed to have the required
properties in \cite{G062}. 
\end{rem}
\begin{deff}
The discrete Fourier transform $\mathcal{F}_m:\mathbb{C}^m\to
\mathbb{C}^m$ is defined by
$$
 \mathcal{F}_m(v)=\bigl( v\cdot _{q^m}\bar e^m_k \bigr)
 _{k=-m/2,\dots,m/2-1},\: v\in \mathbb{C}^m.
$$
\end{deff}
\begin{rem}\label{discrete_parseval}
As the discretization is faithful, $\mathcal{F}_m$ is an isometric
isomorphism. In fact it is easy to prove that
$$
 v\cdot _{q^m}\overline{w} =\mathcal{F}_m(v)\cdot
 \overline{\mathcal{F}_m(w)} 
$$
so that Parseval's identity carries over exactly to the discrete
setting. Notice that the standard Euclidean
inner product is used in the right-hand side.
\end{rem}
\begin{prop}\label{convergence}
For a finite order distribution $u\in\mathcal{D}'_\pi$ and a test
function $\varphi\in\mathcal{D}_\pi$, it can be shown (see
\cite[Theorem 4.2]{G062} and the considerations preceding it) that, given
any $M\in\mathbb{N}$, one has that
$$
 |\langle u,\varphi \rangle-u^m\cdot_{q^m}\varphi ^m|\leq
 c(M,u,\varphi)\frac{1}{m^M}.
$$
\end{prop}
Notice that, while this result is proved in \cite{G062} only for compactly
supported distribution and compactly supported test functions, the
same arguments apply in the current context since the compact support
condition is not needed in the periodic context where no boundary is
present and, hence, no boundary effects (read convergence slowdown due
to boundary mismatch) can occur. 
\begin{rem}\label{kernel truncation}
At the continuous level, one can think of the series
$$
 i(x,y)=\frac{1}{2\pi}\sum _{k\in \mathbb{Z}}e^{ik(x-y)},\: x,y\in[-\pi,\pi),
$$
as the (generalized, since it converges in the sense of distributions
only) kernel $i$ of the identity map on $\operatorname{L}^2_\pi$ since
clearly 
$$
 \varphi (x)=\frac{1}{2\pi}\sum _{k\in \mathbb{Z}}e^{ikx}\int
 _{-\pi}^{\pi}e^{-iky}\varphi(y)\, dy\text{``}=\text{''}\int _{-\pi}^{\pi}i(x,y)\varphi(y)\,
 dy,\: \varphi\in \mathcal{D}_\pi.
$$
In this context, a discretization which respects the duality
pairing and the orthogonality structure as decribed above yields
``natural'' spectral discretizations for a variety of important
operators which will be exploited later. In particular, it delivers such a
discretization $i^m$ for the identity map given by
\begin{multline}\label{identityrep}
  (\delta^m_x|\delta_y^m)_{q^m}:=\langle \delta ^m_x,\bar\delta ^m_y \rangle _{q^m}=\sum
  _{k,\tilde k=-m/2}^{m/2-1}e_k(x)\bar e_{\tilde{k}}(y)\langle \bar
  e_k^m,e^m_{\tilde{k}}\rangle _{q^m}\\= \sum _{k,\tilde{k}=-m/2}^{m/2-1}e_k(x)\bar
  e_{\tilde{k}}(y)\delta _{k \tilde{k}}=\frac{1}{2\pi}\sum
  _{k=-m/2}^{m/2-1}e^{ik(x-y)}=i^m(x,y), 
\end{multline}
which is clearly the truncation of the series representation of the
kernel $i$ itself. Notice that $x,y$ need not be grid points and
that the approximation is thus ``grid blind'' and the error incurred
is caused only by truncation of the series and by evaluation of the
exponential function at the points of interest. If the kernel is
evaluated on the grid points only, then it coincides with the kernel
of discrete identity map, i.e. with the identity matrix
$$
 i^m(x^m_j,x^m_k)=\delta_{jk}.
$$
\end{rem}
\subsection{Higher dimensions}
In higher dimensions, the periodicity box $B=B_N$ is discretized
analogously in each direction by equidistant points to obtain the grid
$$
 G^m=G^m_n=\underset{N\text{-times}}{\underbrace{G^m_1\times\cdots\times
     G^m_1}}, 
$$
with corresponding quadrature rule
$q^m=\bigl(\frac{2\pi}{m}\bigr)^N\mathbf{1}^m$, where now,
$\mathbf{1}^m$ is thought of as a vector of length $m^N$. This way, a
faithful discretization respecting duality pairing and orthogonality
is obtained. In particular, it follows that
$$
 \langle e^m_k,\bar e^m_{\tilde{k}} \rangle _{q^m}=\delta _{k\tilde k},
$$
for $e_k(x)=\frac{1}{(2\pi)^{N/2}}e^{ik\cdot x}$ for $x\in\mathbb{R}^N$,
$k\in\mathbb{Z}^N$ and, again, $e_k^m=e_k\big |_{G^m}$. Dirac delta
functions are approximated by tensor products
$$
 \delta^m_{x_0}=\delta^m_{x^1_0}\otimes\cdots\otimes\delta^m_{x^N_0},
$$
of the corresponding one dimensional representations
$\delta^m_{x^j_0}$, $j=1,\dots,N$ where
$x_0=(x_0^j)_{j=1,\dots,N}$. As far as test functions $\varphi _{x_0}$
supported in a neighborhood of a point $x_0\in B$ go, many choices
can be made. The specifics will be given in the numerical experiments
performed later. For now it is only important to know that such test
functions can be given explicitly by an analytical formula which
allows for accurate evaluation anywhere.

Consider now a general pseudodifferential operator $a(x,D)$ with
symbol $\bigl( a(x,k)\bigr) _{k\in \mathbb{Z}^N}$ defined by
$$
 a(x,D)\varphi =\frac{1}{(2\pi)^{N/2}}\sum _{k\in \mathbb{Z}^N}e^{ik\cdot
   x}a(x,k)\hat\varphi _k=\sum _{k\in \mathbb{Z}^N}e_k(x)a(x,k)\hat\varphi _k,
$$
and where $a(\cdot,k):B\to \mathbb{C}$ is assumed to be smooth and
periodic for each $k\in \mathbb{Z}^N$. Its Schwartz kernel is given by
$$
 k_a(x,y)=\frac{1}{(2\pi)^{N}}\sum _{k\in \mathbb{Z}^N}e^{i
   k\cdot(x-y)}a(x,k)=\sum _{k\in \mathbb{Z}^N}e_k(x)e_k(-y)a(x,k) ,
$$
for which one has that
$$
 a(x,D)\varphi=\langle k_a(x,\cdot),\varphi \rangle\text{ for }\varphi
 \in \mathcal{D} _\pi.
$$
More suggestively one can write that
$$
 k_a(x,y)=\bigl(a(x,D)\delta_y|\delta_x\bigr),
$$
justified by the validity of the formal Parseval's identity
$$
 \bigl(a(x,D)\delta_y|\delta_x\bigr)=\bigl(
 \widehat{a(x,D)\delta_y}|\widehat{\delta_x}\bigr)=\sum
 _{k\in\mathbb{Z}^N}a(x,k)e_k(-y)\overline{e_{-k}(x)}=\sum_{k\in\mathbb{Z}^N}e_k(x)e_k(-y)a(x,k) 
$$
If $a(x,\cdot)$ is polynomially bounded (for each $x$), convergence in
the sense of distributions can be established. For well-known classes
of symbols \cite{Tay96b}, it can be shown that $k_a$ is smooth away
from the diagonal $[x=y]$, where cancellations are responsible for the 
faster convergence of the series. This is the case for general
differential operators and the corresponding solutions operators
appearing in common boundary value problems, for instance. It turns
out that, what was observed above for 
$a\equiv 1$ (leading to the identity map) in one space dimension, is
valid for general pseudodifferential operators.
\begin{thm}\label{discretepseudos}
Given a pseudodifferential operator $a(x,D)$ with kernel $k_a$, it is
natural to approximate it by the truncated series expansion
$$
 k^m_a(x,y)=\frac{1}{(2\pi)^{N/2}}\sum _{k\in \mathbb{Z}^N_m}a(x,k)e^{ik\cdot(x-y)},
$$
where $\mathbb{Z}^N_m=\big\{ k\in \mathbb{Z}^N\, :\,
k_i=-m/2,\dots,m/2-1\text{ for }i=1,\dots,n\big\}$. 
In this case one has that
\begin{equation}\label{kernelrep}
 k^m_a(x,y)=\bigl(a^m(x,D)\delta_y^m|\delta^m_x\bigr)_{q^m},
\end{equation}
for $a^m(x,D)=\mathcal{F}_m^{-1}a^m(x,\cdot)\mathcal{F}_m$ and
$$
 a^m(x,k)=a(x,k),\: k\in \mathbb{Z}^N_m.
$$
\end{thm}
\begin{proof}
In one dimension, the extension of \eqref{identityrep} to general
symbols amounts to
$$
 \bigl(a^m(x,D)\delta_y^m|\delta^m_x\bigr)_{q^m}=\bigl(
 \mathcal{F}_m[a^m(x,D)\delta^m_y]\big
 |\mathcal{F}_m(\delta_x^m)\bigr) _{q^m}
 =\sum _{k\in\mathbb{Z}^N_m}e_k(-y)a(x,k)e_k(x)=k^m_a(x,y),
$$
since the term $\delta _{k\tilde k}$ in \eqref{identityrep} is simply
replaced by $a(x,\tilde k)\delta _{k\tilde k}$. The rest follows from
this and the fact that, in higher dimensions, one has that
\begin{align*}
 \delta_x^m&=\delta^m_{x_1}\otimes\cdots\otimes \delta^m_{x_n}\text{
             and}\\
 e_k(z)&=e_{k_1}(z_1)\otimes\cdots\otimes e_{k_n}(z_n),\: z\in \mathbb{R}^n.
\end{align*}
\end{proof}
This simple observation is quite useful and shows how to produce grid
independent ``spectral'' approximations of operators through an
approximation of their kernels. The structure of the kernel made
apparent in \eqref{kernelrep} provides a blue print as to how to
obtain numerical approximations to kernels of discrete operators
$K^m$ by simply computing $(\delta^m_x|K^m
\delta^m_y)_{q^m}=(K^m \delta^m_y|\delta^m_x)_{q^m}$ (the two coincide
in the real case, which always applies in the examples considered here). This is of 
interest when $K^m$ is, for instance,  the numerical inverse of the
discretization $A^m$ of an operator $A$ for which no analytical inverse
is available.
\begin{rem}
For solution operators, the convergence of the series
can, in general, be quite slow even if it is stronger than in the
sense of distributions. This is due to the (mildly) singular behavior
of the kernel on the diagonal and typically requires special care in
the numerical evaluation process. Representation \eqref{kernelrep},
however, suggests natural ways in which to do this by regularization
of the kernel through
$$ 
 \bigl(\delta^m_x|a^m(x,D)\varphi ^m_{\tilde y}\bigr),
$$
where $\mathcal{D}_\pi\ni\varphi _{\tilde y}\simeq\delta _{\tilde y}$ and $\tilde
y\simeq y$ is conveniently located. In some cases, this modification can
be carried out to obtain an alternate exact representation by a smooth
kernel with no approximation involved. See the boundary value problem
example in the next section.
\end{rem}
\section{Two Dimensional Examples}
Two examples in two space dimensions are presented here which
illustrate the benefits of the proposed approach.
\subsection{Integration}\label{integral}
Consider a (smooth or piecewise smooth) domain $\Omega\subset B$ and
the numerical task of approximating the integral
$$
 I(\varphi)=\int _\Omega \varphi(x)\, dx,\: \varphi\in\mathcal{D}_\pi,
$$
of a smooth function $\varphi$. Since $I$ is a finite order
distribution, one has that
$$
 I=\sum _{k\in\mathbb{Z}^2}\langle I,e_k\rangle \bar e_k=\sum
 _{k\in\mathbb{Z}^2}\tilde{I}_k\bar e_k, 
$$
and $I(\varphi)=\sum _{k\in\mathbb{Z}^2}\tilde{I}_k\hat\varphi _k$. If
$\tilde{I}_k$ can be computed /approximated accurately by
$\tilde{I}^m_k$ on $G^m$, then a numerical quadrature $I^m$ for
integration over $\Omega$ could be obtained by setting
$$
 I^m(\varphi ^m)=\sum _{k\in\mathbb{Z}^2_m}\tilde{I}^m_k\hat
 \varphi^m_k, 
$$
where $\hat\varphi ^m_k=\langle\varphi ^m,\bar e^m_k\rangle
_{q^m}=\mathcal{F}_m(\varphi)_k$ 
can be computed using the Fast Fourier transform. The notation
$\mathbb{Z}^2_m$ is used, as before, for the appropriate set of indeces
corresponding to the discretization level considered. While it appears
that the problem of computing $I(\varphi)$ has simply been replaced by
that of evaluating $I(e_k)$ for $k\in\mathbb{Z}^2_m$, the analytical
knowledge of the bases functions and of their properties becomes
useful. Indeed for $k=(0,0)$ one has
\begin{align}\label{intcoeff1}\notag
 \tilde{I}_0&=\frac{1}{2\pi}\int _\Omega\, dx=\frac{1}{4\pi}\int
 _\Omega\operatorname{div}\begin{pmatrix} x_1\\x_2\end{pmatrix}\,
 dx=\frac{1}{4\pi}\int _\Gamma (x_1\nu _1+x_2\nu _2)\, d\sigma
 _\Gamma(x)\\ &=\frac{1}{4\pi}\int _0^{2\pi}\bigl[ \gamma_1(t)\dot
                \gamma_2(t)-\dot \gamma _1(t)\gamma _2(t)\bigr]\, dt,
\end{align}
where $\Gamma=\partial\Omega$, $\nu$ is the outward unit normal to
$\Gamma$, and $\bigl( \gamma _1(\cdot),\gamma _2(\cdot)\bigr)$ is a
parametrization of $\Gamma$. Here it assumed for simplicity that
$\Gamma$ is connected. If, on the other hand, $k\neq 0$, then
\begin{align}\label{intcoeff2}\notag
 \tilde{I}_k&=\frac{1}{2\pi}\int _\Omega e^{ik\cdot x}\,
  dx=\frac{1}{2\pi |k|^2}\int _\Omega-\Delta e^{ik\cdot x}\,
  dx\\\notag&=-\frac{1}{2\pi |k|^2}\int _\Gamma \nu\cdot\nabla e^{ik\cdot x}\,
  d\sigma _\Gamma=-\frac{i}{2\pi |k|^2}\int _\Gamma
  [k\cdot\nu]e^{ik\cdot x}\,  d\sigma _\Gamma\\&=\frac{i}{2\pi
  |k|^2}\int  _0^{2\pi}\bigl[k_2\dot\gamma _1(t)-k_1\dot
  \gamma _2(t)\bigr]e^{ik\cdot\gamma(t)}\, dt.
\end{align}
Thus, given a representation of $\Omega$ via its boundary $\Gamma$,
either as a list of points (from which the relevant geometric
quantities can be computed) or via an analytic expression (often
available even in pratice), the computation reduces to that of a
periodic one dimensional integral which can be performed to high
accuracy as already noted earlier. The advantage of this approach is
that the integrand lives on $B$, or on $G^m$, and only a simple
discrete representation $\Gamma ^n$ of $\Gamma$ is needed in
order to perform the calculation. Notice that the grids $G^m$ and
$\Gamma ^n$ do not need to have any relation whatsoever to
one another. In fact, when $u$ is smooth, $m$ can be kept small while 
$n$ will need to be chosen large in order to get a good approximation
of the highly oscillatory (in general) line integral. The advantage
clearly lies in the line integral being one dimensional.
\subsection{Boundary Value Problems}
Let again $\Omega$ be a smooth domain inside the box $B$ and consider
the classical boundary value problems
\begin{equation}
\begin{cases}-\Delta u=f&\text{ in }\Omega,\\u=g&\text{ on
  }\Gamma,\end{cases}
\end{equation}
and
\begin{equation}
\qquad \begin{cases}-\Delta u=f&\text{ in }\Omega,\\\partial _\nu
  u=g&\text{ on }\Gamma,\end{cases}  
\end{equation}
where it can be assumed that the data are given as $f:B\to \mathbb{R}$
and $g:\Gamma\to \mathbb{R}$. Using
$$
 G(x,y)=G(x-y)=\frac{1}{2\pi}\log\bigl( |x-y|\bigr)\text{ for }x,y\in
\mathbb{R}^2,
$$
and the classical Green's identity
\begin{equation}\label{green}
\int _\Omega (u\Delta G-G\Delta u)\, dx=\int _\Gamma (u\partial _\nu
G-G\partial _\nu u)\, d\sigma _\Gamma,
\end{equation}
solution representations can be obtained from
\begin{align*}
 u(x)&=\int _\Omega G(x,y)f(y)\, dy-\int _\Gamma G(x,y)\partial _\nu
       u(y)\, d\sigma _\Gamma(y)+\int g(y)\partial _\nu G(x,y) \,
       d\sigma _\Gamma(y),\\
 u(x)&=\int _\Omega G(x,y)f(y)\, dy+\int _\Gamma u(y)\partial _\nu
       G(x,y)\, d\sigma _\Gamma(y)-\int G(x,y)g(y)\,
       d\sigma _\Gamma(y),
\end{align*}
once the boundary functions $u$ and $\partial _\nu u$ are recovered,
depending on whether one considers the Neumann or Dirichlet problem,
respectively. While the single and double layer potentials
terms
\begin{align}\label{players}
 \mathcal{S}(u)(x)&=\int _\Gamma G(x,y)\partial _\nu u(y)\, d\sigma
 _\Gamma(y),\: x\in \mathbb{R}^2 \setminus \Gamma,\\
 \mathcal{D}(u)(x)&=\int _\Gamma u(y)\partial _\nu G(x,y)\, d\sigma
 _\Gamma(y),\: x\in \mathbb{R}^2 \setminus \Gamma.
\end{align}
are important to understand and will appear later for their mapping
properties, the construction of solutions, both analytical and
numerical, presented here will proceed slightly differently. The
following facts \cite{Tay96b} will be useful
\begin{align}\label{jumprelations}
 S(u)(x)&=\lim _{\Gamma\not\ni\tilde x\to x}\mathcal{S}(u)(\tilde
  x)=\int _\Gamma G(x,y)\partial _\nu u(y)\, d\sigma _\Gamma(y),\:
  x\in \Gamma,\\
 \partial _{\nu\pm}S(u)(x)&=\lim _{\Omega ^\pm\ni \tilde x\to
  x}\partial _{\nu(\tilde x)}S(u)(\tilde x)=\mp \frac{1}{2}u(x)+N(u)(x),\: x\in
                \Gamma,
\end{align}
where $\Omega ^+=\Omega$ and $\Omega ^-= \mathbb{R}^2 
\setminus\overline{\Omega}$, respectively, and the normal to $\Gamma$ is
extended continuously in a neighborhood of $\Gamma$, and 
$$
 N(u)(x)=\int _\Gamma u(y)\partial _{\nu(x)} G(x,y)\, d\sigma _\Gamma(y),\:
 x\in \Gamma.
$$
Observe that that the function $G(\cdot,y)$ is clearly a harmonic
function in $\Omega$ for any $y\in B \setminus \Omega$ and for any
fundamental solution $G$. Now consider the Dirichlet problem above and
the shifted Neumann problem given by
\begin{equation}
\qquad \begin{cases}u-\Delta u=f&\text{ in }\Omega,\\\partial _\nu
  u=g&\text{ on }\Gamma,\end{cases}  
\end{equation}
so as to make the problem uniquely solvable. For the Dirichlet problem therefore take
$G^D_\pi(x,y)$ to be the Green's function for the periodicity box $B$
characterized by its symbol 
$$
 \hat G^D_\pi(k)=\begin{cases}0,& k=0,\\
 \frac{1}{|k|^2},&0\neq k\in \mathbb{Z}^2,
 \end{cases}
$$
and, for the Neumann problem, $G^N_\pi$ with symbol
$$
 \hat G^N_\pi(k)=\frac{1}{1+|k|^2},\: k\in \mathbb{Z}^2.
$$
If $f$ is a mean zero function, i.e. if $\hat f_0=0$, then $G^D_\pi*f=\int _B
G^D_\pi(\cdot,y)f(y)\, dy$ satisfies
$$
 -\Delta G^D_\pi*f=f\text{ in }\Omega,
$$
as desired. One also has that
$$
 G^D_\pi*\Bigl(-\Delta u \Bigr)=u -P_0(u), 
$$
where $P_0=(\cdot |e_0)e_0$ is the orthogonal projection onto the
subspace consisting of constant functions. Similarly for the Neumann
problem where
$$
 (1-\Delta)G^N_\pi*f=f\text{ and }G^N_\pi*\bigl( u-\Delta u\bigr)=u.
$$
A solution to the boundary value problems can therefore be sought in
the form 
$$
 u(x)=G^b_\pi*f(x)+\int _\Gamma G^b_\pi(x,y)h(y)\, d\sigma _\Gamma (y),\:
 x\in\Omega,\: b=D,N,
$$
where the second term is a ``harmonic'' function in $\Omega$ and can be
thought of as a superposition along the boundary of functions in the
kernel of $\Delta_\Omega$ or $1-\Delta_\Omega$, respectively, which generate the
desired boundary behavior for the solution. The function $h$ can indeed 
be determined by the requirement that $u=g$ or $\partial _\nu u=g$ on
the boundary $\Gamma$, respectively, that is by insisting that
\begin{equation*}
 g(x)=\langle \delta _x,u\rangle=\langle \delta
  _x,G^D_\pi*f\rangle+\langle \delta _x,G^D_\pi*(h \delta _\Gamma )\rangle,\:
       x\in \Gamma,
\end{equation*}
and that
\begin{equation*}
 g(x)=\langle -\nu(x)\cdot\nabla \delta _x,u\rangle=-\langle
 \partial_{\nu(x)}\delta _x,G^N_\pi*f\rangle -\langle \partial_{\nu_x}\delta _x,
 G^N_\pi*(h\delta _\Gamma )\rangle,\: x\in \Gamma, 
\end{equation*}
where
$$
 H*(h \delta _\Gamma )=\int _\Gamma
 H(x,y)h(y)\, d\sigma _\Gamma(y),
$$
for $H=G^D_\pi,G^N_\pi$. The above is justified by the fact that
\begin{align*}
 \partial_\nu u(x)&=\langle \delta _x,\partial _\nu u\rangle=\langle
 \delta _x,\sum _{j=1}^2\nu _j \partial_j u\rangle=\sum _{j=1}^2\langle \nu _j
 \delta _x,\partial _j u\rangle\\&=\sum _{j=1}^2\langle \nu_j(x)
 \delta _x,\partial _j u\rangle=-\langle\nu(x)\cdot\nabla \delta
 _x,u\rangle\\&=-\langle\partial _{\nu(x)}\delta _x,u\rangle,\: x\in
                                   \Gamma.
\end{align*}
This yields an equation
$$
 M_b(h)=\check g=\begin{cases} g-\langle \delta _\cdot,G^D_\pi*f \rangle,&b=D,\\
 -g-\langle\nu _\cdot\cdot\nabla \delta_\cdot, G^N_\pi*f \rangle,&b=N,\end{cases}
$$
for an operator $M_b$ on $\Gamma$ given by
\begin{equation}\label{M}
 M_b(h)=\int _\Gamma m_b(x,y)h(y)\, d\sigma _\Gamma(y),\: h:\Gamma\to
 \mathbb{R},
\end{equation}
with kernel function defined by
\begin{equation}\label{kerneldn}
 m_b(x,y)=\begin{cases} m_D(x,y)=\langle\delta
   _x,(-\Delta_\pi)^{-1}P_\psi\delta _y\rangle,&b=D,\\
   m_N(x,y)=\langle \partial _{\nu(x)}\delta
   _x,(1-\Delta_\pi)^{-1}\delta _y\rangle,&b=N,\end{cases}\text{ for
 }x,y\in\Gamma,  
\end{equation}
for the Dirichlet and Neumann problem, respectively. Here the more
transparent notation $(-\Delta_\pi)^{-1}$ and $(1-\Delta_\pi)^{-1}$ are
used for the operation of convolution with $G^D_\pi$ and $G^N_\pi$,
respectively. $P_\psi u$ denotes the projection onto mean
zero functions/distributions given by
\begin{equation}\label{adjust}
 P_\psi u=u-\hat u_0\psi=u- \tilde u_0\psi,
\end{equation}
for a nonnegative function $\psi\in \mathcal{D}_\pi$ satisfying
\begin{equation}\label{meanzeroadj}
  \operatorname{supp}(\psi)\subset \Omega^\mathsf{c}\text{ and
  }\hat\psi _0=1.
\end{equation}
\begin{rem}
Using the suggestive notation 
$$
 d\sigma _\Gamma (y)=|\langle dy,\frac{\partial}{\partial t}\rangle|dt,
$$
for $\langle dy,\frac{\partial}{\partial t}\rangle=\dot \gamma(t)$
when $y=\gamma(t)$ to evoke the validity of 
$$
 \int _\Gamma v(y)\, d\sigma_\Gamma (y)=\int _0^{2\pi} v \bigl( \gamma (t)
 \bigr) |\dot \gamma (t)|\, dt,
$$
for any parametrization $\gamma$ of $\Gamma$ and for any smooth
integrand $v:\Gamma\to \mathbb{R}$, allows for the factor $|\langle
dy,\frac{\partial}{\partial t}\rangle|$ to be assimilated 
into the unknown function $h$ to yield $|\langle
dy,\frac{\partial}{\partial t}\rangle|h$ as the new 
unknown. This is particularly convenient when working at the discrete
level, where one is only eventually interested in the function
$$
 x\mapsto \int _\Gamma G ^b_\pi(x,y)h(y)\, d \sigma_\Gamma(y) =\int
 _0^{2\pi} G^b _\pi(x,y)h(y)|\langle dy,\frac{\partial}{\partial
   t}\rangle|\, dt,\: \Omega\to\mathbb{R}\,  
$$
and the determination of $h$ or $|\langle
dy,\frac{\partial}{\partial t}\rangle|h$ are equivalent.
\end{rem} 
The kernels $m_D$ and $m_N$ in \eqref{kerneldn} have the form
of those considered in the previous section and are of
exactly the same type as in the earlier one dimensional toy boundary
value problem. Just as in that case, $\delta_y$ can be replaced by
$\delta_{\tilde{y}}$ for
$$
 \tilde{y}=y+\delta \nu(y),\: y\in \Gamma,
$$
where $\delta>0$ can be chosen such that a tubular neighborhood 
$$
 T^\delta _\Gamma=\{ x\in B\, |\, d(x,\Gamma)< 2 \delta\}
$$
of $\Gamma$ can be found with well-defined coordinates $(y,s)\in
\Gamma\times(-2\delta,2 \delta)$ satisfying
$$
 x=y+s\nu(y)\text{ for }y=Y(x)\text{ and } s=d(x,\Gamma),
$$
where $Y(x)$ denotes the point on $\Gamma$ closest to $x$.
This corresponds to replacing $\Gamma$ by $\widetilde{\Gamma}=\{
\tilde{y}\, |\, y\in \Gamma\}$ in the evaluation of the kernel (but
not in that of the boundary integral). Notice that latter distinction is
immaterial at the discrete level where the boundary measure is
assimilated in the unknown function $h$ as described above. An even
better choice is obtained by replacing $\delta _y$ by
\begin{equation}\label{tfct}
 \varphi _{\tilde{y}}\in \mathcal{D}_0\text{ with
 }\operatorname{supp}(\varphi _{\tilde{y}})\subset \Omega
 ^\mathsf{c}\text{ and }\operatorname{supp}(\varphi_{\tilde y})\cap
 \operatorname{supp}(\psi)=\emptyset.
\end{equation}
The kernel modification is shown pictorially in Figure \ref{Figure:newkernel}.
\begin{center}
\begin{figure}
\begin{tikzpicture}
 \draw[domain=-10:100] plot[samples=100] ({2+3*cos(\x)}, {3
   *sin(\x)});
 \filldraw ({2+3*sqrt(2)/2},{3*sqrt(2)/2}) circle (1pt);
 \draw[dotted,thick] ({2+3.6*sqrt(2)/2},{3.6*sqrt(2)/2}) circle
(12pt);
 \node at ({2+3*sqrt(2)/2-0.2},{3*sqrt(2)/2-0.2}) {$y$};
 \node at (2,3.2) {$\Gamma$};
 \node at (3,1) {$\Omega$};
 \node at ({2+3.6*sqrt(2)/2},{3.6*sqrt(2)/2+1.2})
 {$\operatorname{supp}(\varphi _{\tilde y})$};
 \draw[->]
 ({2+3.6*sqrt(2)/2},{3.6*sqrt(2)/2+1})--({2+3.6*sqrt(2)/2},{3.6*sqrt(2)/2+0.45});
 \draw[->]
 ({2+3*sqrt(2)/2},{3*sqrt(2)/2})--({2+4.5*sqrt(2)/2},{4.5*sqrt(2)/2}); 
 \filldraw ({2+3.6*sqrt(2)/2},{3.6*sqrt(2)/2}) circle (1pt);
 \node[right] at ({2+3.6*sqrt(2)/2},{3.6*sqrt(2)/2-0.1})
 {$\tilde y$};
 \node[right] at ({2.05+4.5*sqrt(2)/2},{4.5*sqrt(2)/2}) {$\nu _{\Gamma}(y)$};
 \node at (-1.5,2) {$G_\pi(x,y)=\langle \delta_x,(-\triangle _\pi)^{-1} P_\psi
   \delta_y\rangle$};
 \node[rotate=90] at (-1.5,1.5) {$\simeq$};
 \node at (-1.5,1) {$\widetilde{G}_\pi(x,y)=\langle
   \delta_x,(-\triangle _\pi)^{-1} P_\psi \varphi _{\tilde
       y}\rangle$};
\node at (-1.5,0) {$\varphi _{\tilde y}\simeq
    \delta_{\tilde y}\simeq \delta_y$};
\end{tikzpicture}
\caption{\label{Figure:newkernel} A pictorial illustration of the proposed kernel construction.}
\end{figure}
\end{center} 
The upshot is that the operator $M_b$ with singular kernel is
replaced by the operator $\widetilde{M}_b$ with smooth kernel given by
$$\begin{cases}
 \widetilde{m}_D(x,y)=\langle \delta _x,(-\Delta_\pi)^{-1} P_\psi\varphi
 _{\tilde{y}}\rangle,&b=D,\\
 \widetilde{m}_N(x,y)=\langle \partial _{\nu(x)}\delta
 _x,(1-\Delta_\pi)^{-1}\varphi _{\tilde{y}}\rangle,&b=N,
\end{cases}\text{ for }x,y\in \Gamma.
$$
By choosing $\varphi _{\tilde{y}}$ localized enough (read close to a
Dirac delta function) it follows that
$$
 M_b-\widetilde{M}_b\simeq 0,
$$
in the strong operator sense. Notice that the projection procedure
\eqref{adjust} ensures that the support of $P_\psi\varphi
_{\tilde{ y}}$ lies completely outside of $\Omega$ and does thus still
generate functions in the kernel of $\Delta_\Omega$.
\begin{rem}
The operator $M_b$ can be shown to be smoothing of one degree of
differentiability in the Dirichlet case, and of none in the Neumann
case. For a proof based on symbol analysis see e.g. \cite{Tay96b}.
\end{rem}
\begin{rem}
While it is often convenient to work with an explicit fundamental
solution for $-\Delta$ and use it in order to derive the necessary
boundary kernels (to be used in a numerical implementation of boundary
integral type), the approach described above does not rely on the
explicit knowledge of a Green's function. Indeed at the discrete
level, the kernel functions, $G^D_\pi=(-\Delta _\pi)^{-1}$ and
$G^N_\pi=(1-\Delta_\pi)^{-1}$ in the examples, 
can be replaced by $(A^m_D)^{-1}$ and $(A^m_N)^{-1}$ for any
discretizations $A^m_b$  to the grid $G^m$ of a differential
operator $A_{b}$ obtained by spectral or finite difference
methods for $b=D,N$. In the above example $A^m$ would be a 
standard spectral or finite difference approximations of the periodic
$-\Delta$ and $1-\Delta$ operators on the box $B$. This opens the
door to applying the method 
to nonconstant coefficient operators and to constant coefficient
operators for which no explicit Green's function or symbol is available.
\end{rem}
Next an illustrative analytical result is proved in the Dirichlet
case which will play an important role in obtaining invertibility
results for the numerical schemes derived later.
\begin{lem}
The operator $M_D$ defined in \eqref{M} with kernel $m_D$ given by
\eqref{kerneldn} is invertible.
\end{lem}
\begin{proof}
First notice that $G^D_\pi$ is a fundamental solution on the space of mean
zero distributions. It follows either from Poisson's summation formula
or from the theory of pseudodifferential operators \cite{Tay96b} that
$G^D_\pi$ is smooth away from the diagonal $[x=y]$ and that 
$$
 G^D_\pi(x,y)\simeq \frac{1}{2\pi}\log \bigl( |x-y| \bigr)=G(x,y),\:
 x\simeq y\in B,
$$
i.e., it has the same singular behavior of the full space fundamental
solution $G$. It indeed differs from it by a smooth kernel only. 
Now one has that
$$
 m_D(x,y)=\frac{1}{4\pi^2}\sum_{k\in \mathbb{Z}^2}e^{ik\cdot
 x}\frac{1}{|k|^2}\bigl( e^{-ik\cdot y}-\hat\psi _k \bigr)=\frac{1}{4\pi^2}
 \sum_{0\neq k\in\mathbb{Z}^2}e^{ik\cdot x}\frac{1}{|k|^2}\bigl( 
 e^{-ik\cdot y}-\hat\psi _k \bigr)
$$
and thus that
$$
 m_D(x,y)=G^D_\pi(x,y)-\frac{1}{4\pi^2}\sum_{0\neq k\in\mathbb{Z}^2}e^{ik\cdot
   x}\frac{\hat\psi _k}{|k|^2}=G^D_\pi(x,y)-\eta(x),
$$
where $\eta$ is a smooth function. Consequently one sees that
$$
 \widetilde{S}(h)=\int _\Gamma G^D_\pi(\cdot,y)h(y)\, d
 \sigma_\Gamma(y)-\eta(x)\int _\Gamma h(y)\, d
 \sigma_\Gamma(y)=S_\pi(h)-\eta(x)\int _\Gamma h(y)\, d
 \sigma_\Gamma(y).
$$
This means that $\widetilde{S}$ enjoys the same classical jump
relations as $S$ (and $S_\pi$) given by
$$
 S(h)=\int _\Gamma G(\cdot,y)h(y)\, d \sigma_\Gamma(y),
$$
i.e. it holds
\begin{equation}\label{jumprelation}
 \begin{cases}
  \gamma^+_\gamma\widetilde{S}(h)=\gamma^-
  _\gamma\widetilde{S}(h)=\gamma_\Gamma\widetilde{S}(h),&\\
  \partial_{\nu^+_\Gamma}\widetilde{S}(h)-\partial
  _{\nu^-_\Gamma}\widetilde{S}(h)=-h,&
 \end{cases}
\end{equation}
where the superscripts $\pm$ indicate limits taken from within and
from without $\Omega$, respectively, just as in
\eqref{jumprelations}. It also follows (see e.g. \cite{Tay96b}) that $M_D$ is
Fredholm and that it continuously maps 
$\operatorname{H}^s(\Gamma)$ to $\operatorname{H}^{s+1}(\Gamma)$ for
any $s\in \mathbb{R}$. It is therefore enough to show
that $M_D$ is injective ``on smooth functions'', i.e. that 
$$
 \gamma _\Gamma \widetilde{S}(h)=\gamma_\Gamma\int _\Gamma 
 m_D(\cdot,y)h(y)\, d \sigma_\Gamma (y)=0\Longrightarrow h\equiv 0,  
$$
for smooth $h:\Gamma\to \mathbb{R}$. Since $\widetilde{S}(h)$ is
defined for all $x\in B$ and is harmonic in $B \setminus \Gamma$,
unique solvability of the Dirichlet problem in $\Omega$ yields that
$\widetilde{S}(h)\big |_{\Omega}\equiv 0$. It follows that
$$
 \partial_{\nu_\Gamma^+}\widetilde{S}(h)-\partial_{\nu_\Gamma^-}\widetilde{S}(h)
 =-\partial_{\nu_\Gamma^-}\widetilde{S}(h)=-h.
$$
Now, in $B \setminus \Omega $ one has that
$$
 \Delta \widetilde{S}(h)=\bigl(\psi-\hat\psi_0e_0\bigr)\int _\Gamma h(y)\, d
 \sigma _\Gamma(y),
$$
and, consequently, that
$$
 -\int _\Gamma h(y)\, d\sigma _\Gamma(y)=\int
 _\Gamma  \partial_{\nu_\Gamma^-}\widetilde{S}(h)d \sigma 
 _\Gamma(y)= \int _{B \setminus \Omega}\Delta \widetilde{S}(h)=
\hat\psi_0\bigl[ 1-\frac{|B \setminus
\Omega|}{4\pi^2}\bigr]\int _\Gamma h(y)\, d\sigma _\Gamma(y),
$$
since $\operatorname{supp}(\psi)\subset B \setminus \Omega$ and
$\hat\psi _0=1$. This,
in turn, implies that $\int _\Gamma h(y)\, d\sigma _\Gamma(y)=0$ because
$|B \setminus\Omega|<|B|=4\pi^2$. For such a $h$, it therefore
holds that
$$
 \widetilde{S}(h)=S_\pi(h).
$$ 
By construction it holds that
$$
 \int _B {S_\pi}(h)\, dx=0,
$$
so that Poincar\'e's inequality yields
$$
 \int _{B \setminus \Omega}S_\pi(h)^2\, dx=\int _B
 S_\pi(h)^2\, dx\\\leq c\int _B \big
 |\nabla S_\pi(h)\big |^2\, dx= c\int _{B \setminus
   \Omega}\big |\nabla S_\pi(h)\big |^2\, dx,
$$
and entails that, if $S_\pi(h) \big | _{B \setminus \Omega}$ is
constant, then it has to vanish identically. Since
$$
 0=-\int _{B \setminus \Omega}S_\pi(h)\Delta
 S_\pi(h)\, dx=\int _{B \setminus \Omega}\big |
 \nabla S_\pi(h)\big |^2\, dx+\int _\Gamma
 \underset{=0}{\underbrace{S_\pi(h)}}\,\partial _{\nu_\Gamma}
   S_\pi(h)\, d \sigma_\Gamma,
$$
it therefore follows that $S_\pi(h) \big | _{B \setminus
\Omega}\equiv 0$. Finally this shows that
$$
 \partial _{\nu _\Gamma^-}S_\pi(h)\big | _{B \setminus \Omega}=h=0,
$$
thus establishing the claim.
\end{proof}
\begin{prop}\label{invtbility}
The modified operator $\widetilde{M}_D$ is injective provided
$\tilde{y}\simeq y$ and $\varphi _{\tilde{y}}\simeq
\delta_{\tilde{y}}$ for $y\in \Gamma$. 
\end{prop}
\begin{proof}
The operator $\widetilde{M}^D$ has smooth kernel and is therefore
compact. Given any smooth $h\not\equiv 0$, it follows from the
previous lemma that $\gamma_\Gamma\widetilde{S}(h)\not\equiv 0$. Now
it holds that
$$
 \langle \delta _x,(-\Delta _\pi)^{-1}P_\psi{\varphi}
 _{\tilde{y}}\rangle\to\langle  \delta _x,(-\Delta
 _\pi)^{-1}P_\psi{\delta}
 _{\tilde{y}}\rangle\text{ as }\varphi_{\tilde y}\to {\delta} _{\tilde y},
$$
pointwise everywhere in $x,y\in \Gamma$ (in fact, uniformly). On the
other hand, one also has that
$$
 \delta_{\tilde{y}}\to \delta_y\text{ as }\tilde{y}\to y,
$$
uniformly in $y\in \Gamma$ in the sense of distributions (or in the
sense of measures) so that
$$
  \langle \delta _x,(-\Delta _\pi)^{-1}P_\psi\delta _{\tilde{y}}\rangle\to
 \langle  \delta _x,(-\Delta _\pi)^{-1}P_\psi\delta
 _y\rangle\text{ as }\tilde y\to y,
$$
pointwise for $x\neq y$, i.e., almost everywhere. Since the limiting
kernel is integrable in view of its logarithmic behavior in the
singularity and provides a bound for the approximating kernels,
Lebesgue's theorem yields that
$$
 \langle \delta _x,(-\Delta _\pi)^{-1}P_\psi\delta _{\tilde{y}}\rangle\to
 m_D(x,y)\text{ in }\operatorname{L}^1 \bigl(\Gamma,d
 \sigma_\Gamma(y)\bigr),
$$
uniformly in $x\in \Gamma$, and, in fact, uniformly in
$|x-y|\geq\varepsilon$ for any $\varepsilon>0$. Consequently
$$
 \widetilde{M}_D(h)\to M_D(h)\text{ as }\tilde{\Gamma}\to \Gamma,
$$
uniformly in $\bigl[\| h\| _2=1\bigr]$ due to the mild (in particular
square integrable) singularity of $m_D$ on the diagonal. This then
entails that  
$$
 \operatorname{ker}(\widetilde{M}_D)=\{ 0\},
$$
for $\tilde{\Gamma}$ close enough to $\Gamma$.
\end{proof}
This useful property will remain valid after discretization, which is
just an additional approximation, even if, as will be demostrated in 
the numerical examples, the modified and the original boundary are 
not that close to each other.
\begin{rem}
The result shows that the functions $w_{\tilde y}:\Omega\to
\mathbb{R}$ given by
$$
 w_{\tilde y}(x)=\int _B G^D_\pi(x,z)P_\psi\varphi _{\tilde y}(z)\, dz\text{
   for } y\in \Gamma 
$$
are ``linearly independent'' elements of $\operatorname{ker}(\Delta
_\Omega)$ if $\widetilde{\Gamma}\simeq \Gamma$ and $ \varphi _{\tilde y}\simeq
\delta_{\tilde y}$. This is intuitively clear for $\widetilde{\Gamma}=\Gamma$
and $\varphi _{\tilde y}\simeq\delta_y$ since, then, $w_y$ are
functions with singularities at different locations $x=y$, yielding a
``diagonally dominant'' kernel (or matrix, at the discrete level).
\end{rem}
\begin{rem}
When dealing with the Neumann problem in the classical way, the fact
that the normal derivative of $S$ is not continuous across $\Gamma$ as
clearly indicated by \eqref{jumprelations}, does require care in
obtaining the correct numerical formulation. By using the kernel
generation procedure described in this paper, however, the problem is
completely avoided, since the relevant kernel $\widetilde{m}_N$ is
smooth thanks to the replacement of $\delta_y$ by $\varphi _{\tilde
  y}$ in its construction.
\end{rem}
\begin{rem}
Notice that the proposed kernel construction effectively replaces a
pseudo-differential operator of type $-1$ or type $0$ for $b=D$ or
$b=N$, respectively, with an infinitely smoothing
operator. Incidentally, an operator of type $k$ is a bounded linear
operator which maps, in the above context,
$\operatorname{L}^2(\Gamma)$ to $\operatorname{H}^{-k}(\Gamma)$. This
has important consequences. One is that the approximating operator is
compact with unbounded inverse (more so that the approximated
operator), i.e. it does not enjoy the same ``functional'' mapping
properties. At the numerical level this will be reflected in a
significant increase in the condition number of the discretized
operator. It will, however, be possible to use natural ``rougher''
discretizations of the same operator as preconditioners, thus
completely curing the conditioning issues, while maintaining the
highly desirable fast converging numerical discretizations to the
approximate smooth kernel and, consequently, accuracy.
\end{rem}
\section{Numerical Implementation and experiments}
The periodc box $B=[-\pi,\pi]^2$ is discretized by a uniform grid
$G^m$ of $m^2$ points by discretizing each direction by
$$
 z^m_j=-\pi+2\pi\frac{j}{m},\: j=0,\dots,m-1,
$$
where $z=x_1,x_2$. The boundary value problems will be posed on the
unit circle centered at the origin, i.e. $\Omega=\mathbb{B}(0,1)$ and the 
padding function $\psi$ of \eqref{meanzeroadj} is defined by
$$
 \psi(x_1,x_2)=e^{-200\sin^2[\frac{1}{2}(x_1-\pi)]\sin^2[\frac{1}{2}(x_2-\pi)]}.
$$
While it is not analytically compactly supported away from $\Omega$,
it numerically vanishes outside a neighborhood of the boundary of the
periodicity box $B$ as show in the contourplot below. The boundary of the
domain $\Omega$ is discretized by $n$ equidistant points
$$
 y_j=\bigl(\cos(\theta _j),\sin(\theta _j)\bigr),\: j=0,...,n-1,
$$
where $\theta _j=2\pi \frac{j}{n}$, yielding the set
$\Gamma^n$. Wherever required, the analytical 
knowledge of the boundary $\Gamma$ of $\Omega$ will be used to
obtain numerical quantities such as, e.g., normal and tangent
vectors. In some applications these might need to be replaced by their
numerical counterparts or done away with altogether by choosing as
centers for the required test-functions points on the grid which are
roughly located along the (numerical) outward normal.
\begin{center}
\includegraphics[scale=0.35]{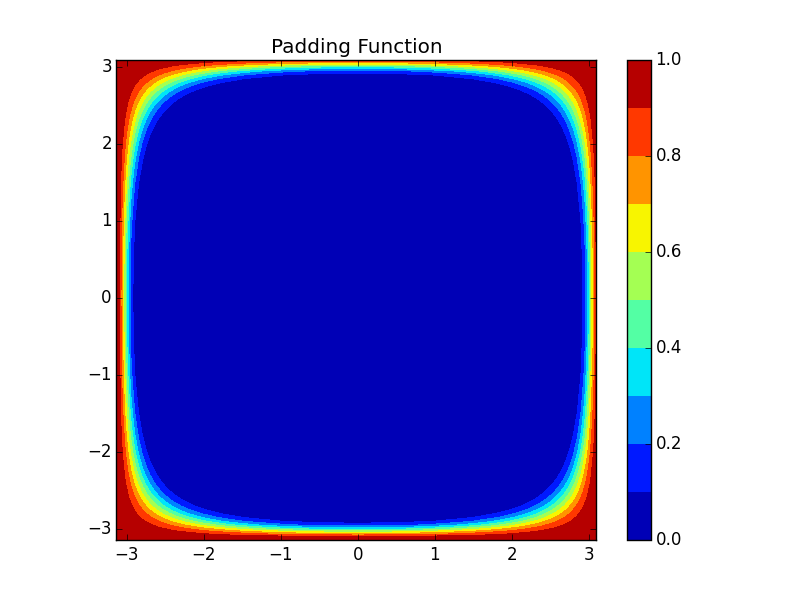}
\includegraphics[scale=0.35]{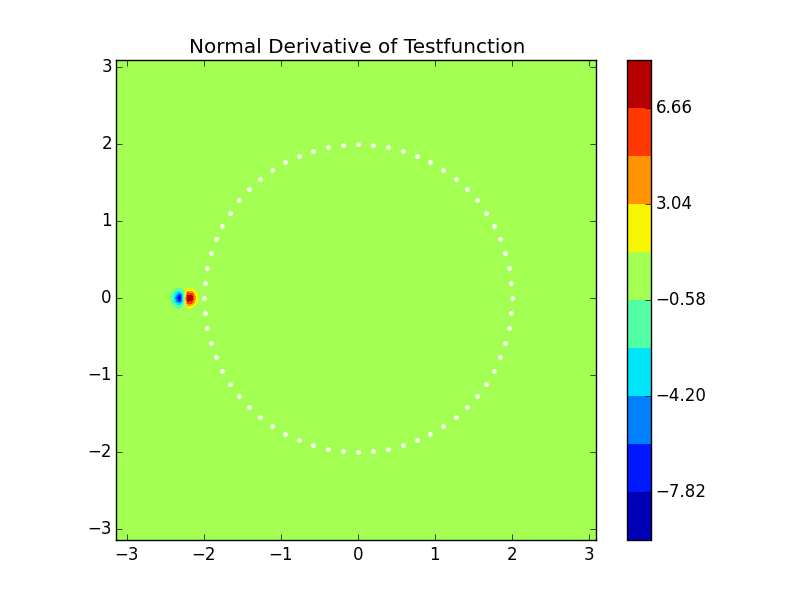}
\end{center}
At the chosen discretization level $m$, the discrete Laplace
operator $-\triangle ^m$ on the periodicity box is represented spectrally
via discrete Fast Fourier transform $\mathcal{F}_m$ via
$$
 \mathcal{F} _m^{-1}\operatorname{diag}\bigl[ (|k|^2)_{k\in
   \mathbb{Z}^2_{m}}\bigr]\mathcal{F} _m.
$$
The projection $P_\psi$ of \eqref{adjust} is discretized by
$$
 P^m_\psi(u^m)=u^m-\frac{\mathcal{F} _m(u^m)(0,0)}{\mathcal{F} _m(\psi
     ^m)(0,0)}\psi ^m,
$$
where $u^m$ is a grid vector, i.e. a function defined on the grid $G^m$
and $\psi ^m$ is the evaluation of $\psi$ on it. The testfunctions
$\varphi _{\tilde y}$ supported about the point $\tilde y\in B
\setminus \Omega$ used in the set up of the kernel are chosen of two
different types: symmetric and non-symmetric. The former are defined
through
$$
 \varphi _{\tilde y}(z)=e^{-\alpha\sin^2[\frac{1}{2}(z_1-\tilde
     y_1)]\sin^2[\frac{1}{2}(z_2-\tilde y_2)]}, \: z\in B
$$
and are discretized by evaluation on the grid $G^m$ and setting
$\alpha = 4m$ in order to make the testfunction ``sharper''
compatibly with the resolution power of the grid. For reasons to be
explained later, non-symmetric and ``sharper'' testfunctions are
useful. Given a point $y\in \Gamma=\mathbb{S}^1$, let $\tau=\tau(y)$ and
$\nu=\nu(y)$ denote the corresponding unit tangent and normal vector, 
respectively. Then consider
\begin{equation}\label{dtfct}
 \partial _\nu \varphi _{\tilde y},
\end{equation}
where the reader is reminded that
$$
 \tilde y=y+\delta \nu(y),\: y\in \Gamma.
$$
This type of testfunction, depicted in the contourplot above, has the
added adavantage of automatically having vanishing average, and plays
an important role in deriving efficient numerical discretizations (see
Subsection \ref{constcoeff}). 
\subsection{Bulk Integrals}
As a first example consider the domain integral as described in
Section \ref{integral}. Letting $\Omega=\mathbb{B}(0,2)$ and computing
the Fourier coefficients $\tilde I_k$ of the distribution $I=\chi _\Omega$ just as
explained in \eqref{intcoeff1}-\eqref{intcoeff2} by using the trapezoidal rule for the angular
parametrization of $\mathbb{S}^2_2$, one obtains a quadrature rule for
integration over $\Omega$. Table \ref{Table:integral} summarizes the
results obtained when applying the quadrature to the function
$$
 u=\cos(\frac{\pi}{4} r^2),\: r=|x|>0.
$$
\begin{center}
\begin{table}
\caption{Relative error for the Fourier quadrature rule at different
  discretization levels.}
\label{Table:integral}
\begin{tabular}{|ccc|ccc|}
\hline
$m$&$n$&$e^{m,n}$&$m$&$n$&$e^{m,n}$\\\hline
32&128& 3.78e-03&128& 128& 4.30e-08\\
& 256& 3.78e-03&& 256& 4.33e-08\\
& 512& 3.78e-03&& 512& 4.33e-08\\
64& 128& 1.80e-04&192& 128& 3.07e-07\\
& 256& 1.80e-04&& 256& 2.58e-07\\
& 512& 1.80e-04&& 512& 2.58e-07\\
96& 128& 8.73e-06&256& 128& 1.81e-08\\
& 256& 8.73e-06&& 256& 4.05e-08\\
& 512& 8.73e-06&& 512& 4.05e-08\\\hline
\end{tabular}
\end{table}
\end{center}
It appears that the number of discretization points $n$ has less of an
impact on the accuracy than the bulk discretization level $m$ as
can be expected since the integrand is radially symmetric. 
\subsection{Dirichlet Problem}
Consider now the homogeneous Dirichlet Problem on $\mathbb{B}(0,1)$ and take the
right hand side to be $f\equiv 1$ defined on whole square $B$. In a
first step, a grid vector $v^m$ is determined satisfying 
$$
 -\triangle ^m v^m\equiv 1.
$$
This can be done simply by taking
$$
 v^m=\mathcal{F}^{-1}_m\operatorname{diag}\bigl[ (g^m_\pi(k))_{k\in
   \mathbb{Z}^2_{m}}\bigr]\mathcal{F}_m\bigl(
 P^m_\psi(1^m)\bigr)=:G^m_\pi(1^m), 
$$
where $1^m$ is the constant grid function with value $1$ and
$$
 g^m_\pi(k)=\begin{cases} 0,&\text{if }k=(0,0),\\ |k|^{-2},&\text{if
   }k\in \mathbb{Z}^2_m \setminus\{(0,0)\}.
 \end{cases}
$$
Next the boundary weight vector $w^n$ is determined such that
$$
 \big\langle \delta ^m_{y_j}, v^m+\sum _{k=1}^nw^n_kG^m_\pi(\varphi
 ^m_{\tilde y_k})\big\rangle _{q^m}=0\text{ for }j=1,\dots,n.
$$
This leads to a system of equations for the entries of $w^n$
characterized by the matrix $M$ with entries
$$
 M_{jk}=\big\langle \delta ^m_{y_j},G^m_\pi(\varphi ^m_{\tilde
   y_k})\big\rangle _{q^m},\: j,k=1,\dots,n,
$$
following the blueprint laid out in the previous section. It can be
viewed as being close to the spectral discretization
$$
 k^m(x,y)=\langle \delta ^m_x,(-\triangle ^m)^{-1}P^m_\psi(\delta
 ^m_{\tilde y})\rangle _{q^m},\: x,y\in \Gamma^n\subset \mathbb{S}^1.
$$
of the smooth kernel
$$
 k(x,y)=\langle \delta _x,(-\triangle)^{-1}P_\psi(\delta _{\tilde
   y})\rangle,\: x,y\in \mathbb{S}^1.
$$
As mentioned earlier this discretization $k^m$ is actually independent
of the grid $G^m$ and can be evaluated anywhere in $B\times B$, in
particular on $\Gamma ^n\times \Gamma^n$. It follows from Proposition
\ref{invtbility} that $M$ is invertible for appropriate choices of
$\tilde y$ for $y\in \Gamma^n$ and of testfunctions $\varphi _{\tilde
  y}$. Once the grid vector $w^m$ is found, a numerical solution of
the Dirichlet problem is given by
$$
 r^m_\Omega u^{m,n} =r^m_\Omega \bigl(v^m+\sum _{k=1}^nw^n_kG^m_\pi(\varphi
 ^m_{\tilde y_k})\bigr),
$$
where $r^m_\Omega$ denotes the restriction (of functions defined on
$B$ or of vectors defined on the grid $G^m$) to $G^m\cap \Omega
$. The numerical results presented
in Table \ref{Table:Dirichlet} provide information about the relative
$l_2$ and $l_\infty$ 
errors $e^{m,n}_2$ and $e^{m,n}_\infty$ computed as follows
$$
 e^{m,n}_p=\frac{\| r^m_\Omega u^{m,n}-r^m_\Omega u\|
   _{l_p}}{\|r^m_\Omega u\| _{l_p}}\text{ for }p=2,\infty,
$$
This is done for various combined discretization 
levels $(m,n)$, various distances of $\tilde y$ from $y\in \Gamma^n$,
and types of testfunctions in Tables
\ref{Table:Dirichlet}--\ref{Table:Diracs}. Recorded is also the
condition number of 
the obtained matrix $M$. The results with fixed distance $\delta=0.4$
are summarized in Table \ref{Table:Dirichlet}.
\begin{table}[tp]
\caption{Numerical Results for the Dirichlet Problem, $\delta=0.4$}
\label{Table:Dirichlet}
\begin{tabular}{|ccccc|ccccc|}
\hline
$m$&$n$&$e^{m,n}_\infty$&$e^{m,n}_2$&$\operatorname{cond}(M)$&
$m$&$n$&$e^{m,n}_\infty$&$e^{m,n}_2$&$\operatorname{cond}(M)$\\\hline
64&64&2.12e-05&3.18e-05&2.6e+03 &512& 128& 2.17e-11& 1.71e-11& 1.7e+06\\
    &80&1.33e-05&2.90e-05&1.7e+04&& 144& 1.61e-11& 4.83e-12& 8.3e+06\\
    &96&1.10e-05&2.75e-05&1.1e+05&& 160& 3.01e-11& 7.16e-12& 3.9e+07\\
    &112&6.11e-05&1.42e-04&1.1e+08 &1024& 64& 1.30e-06& 1.10e-06& 2.5e+03\\ 
128&64&1.03e-06&1.01e-06&2.5e+03 & &80& 5.44e-08& 4.74e-08& 1.3e+04\\
& 80& 1.94e-07& 1.62e-07& 1.3e+04& &96& 2.43e-09& 2.14e-09& 6.9e+04\\
& 96& 2.07e-07& 9.92e-08& 7.0e+04 & &112& 1.14e-10& 9.89e-11& 3.5e+05\\
256& 64& 1.26e-06& 1.10e-06& 2.5e+03 & &128& 5.48e-12& 4.68e-12& 1.7e+06\\
& 80& 5.43e-08& 4.74e-08& 1.3e+04&&144& 2.75e-13& 2.24e-13& 8.3e+06\\
& 96& 2.32e-09& 2.13e-09& 6.9e+04& &160& 2.38e-14& 6.39e-15& 3.9e+07\\
& 112& 1.17e-10& 1.01e-10& 3.5e+05 && 176& 2.44e-14& 6.30e-15& 1.8e+08\\
\hline
\end{tabular}
\end{table}
It appears clearly that accuracy tends to grow for a given grid
parameter $m$ with increasing number of boundary discretization points
$n$. This happens until the boundary discretization becomes too fine
compared to the given, fixed discretization of the periodicity
box. Notice that, if the parameter $n$ is kept 
fixed, the accuracy improves also as a function of the discretization
size $m$. Similarly gains stop accruing when the box discretization
becomes too fine compared to the fixed boundary resolution. As the
operator approximated by $M$ is of negative order $1$, the condition
number of $M$ is expected to grow linearly in the discretization
size. Indeed increasing $n$ enlarges the condition number. This effect
is, however, compounded by the matrix $M$ becoming less 
and less diagonally dominant as the boundary discretization points
become denser while the support of the testfunctions remains unchanged
for fixed discretization level $m$. Notice that, for fixed $n$, the
condition number of $M$ remains virtually unchanged as $m$
changes. The ``optimal'' value (for the specific choice of
testfunction type and support size) was chosen based on the results
found in Table \ref{Table:Distance} where the arbitrary but still
representative choice of $m=256$ is made and a variety of discretization
levels $n$ are shown. The distance is steadily increased until it no
longer leads to an improvement in the approximation quality. It can be
seen that the accuracy improves with distance and that optimal
distance decreases as the box discretization gets finer, thus allowing
for a stronger resolution power and, consequently, a better
approximation of the testfunctions. There appears to be a trade-off
between condition number of $M$ and accuracy of the outcome, where the
best accuracy is obtained at the cost of a high condition number.
\begin{table}[tp]
\caption{Dependence on
  $\delta=\operatorname{dist}(\Gamma,\widetilde{\Gamma})$ for $m=8$}
\label{Table:Distance}
\begin{tabular}{|ccccc|ccccc|}
\hline
$n$&$\delta$&$e^{m,n}_\infty$&$e^{m,n}_2$&$\operatorname{cond}(M)$&
$n$&$\delta$&$e^{m,n}_\infty$&$e^{m,n}_2$&$\operatorname{cond}(M)$\\
\hline
64& 0.15 & 9.06e-04& 8.69e-04& 9.6e+01 & 80& 0.5&2.59e-09& 2.09e-09& 6.2e+04\\
& 0.2 & 2.28e-04& 2.18e-04& 1.9e+02&& 0.6 &1.54e-10& 1.05e-10& 2.7e+05\\
& 0.3 & 1.58e-05& 1.45e-05& 7.1e+02&& 0.7 &3.84e-11& 1.12e-11& 1.2e+06\\
& 0.4 & 1.26e-06& 1.10e-06& 2.5e+03&& 0.8 &2.95e-11& 1.06e-11& 4.9e+06\\
& 0.5& 1.15e-07& 9.33e-08& 8.3e+03 &144&0.15&6.10e-06& 1.02e-05& 3.9e+03\\
& 0.6& 1.20e-08& 8.90e-09& 2.7e+04 &&0.2& 4.13e-08& 3.81e-08& 1.9e+04\\
& 0.7 & 1.42e-09& 9.41e-10& 8.6e+04&& 0.3& 1.01e-10& 1.02e-10& 4.3e+05\\
& 0.8 & 1.80e-10& 9.65e-11& 2.6e+05&& 0.4& 1.61e-11& 4.83e-12& 8.3e+06\\
& 0.9 & 6.59e-11& 2.09e-11& 7.9e+05&192&0.15& 2.85e-08& 1.40e-08& 2.9e+04\\
80& 0.15 &2.42e-04& 2.06e-04& 2.1e+02&& 0.2& 3.95e-10& 3.63e-10&2.5e+05\\
& 0.2& 4.24e-05& 3.77e-05& 5.1e+02&& 0.3& 1.04e-13& 7.99e-14& 1.6e+07\\
& 0.3& 1.40e-06& 1.24e-06& 2.7e+03&& 0.35& 1.94e-14& 2.96e-15& 1.3e+08\\
& 0.4& 5.43e-08& 4.74e-08& 1.3e+04 &&&&&\\\hline
\end{tabular}
\end{table}
In perfect agreement with the theoretical analysis, the condition
number of $M$ is the least when using Dirac delta functions located
along the discrete boundary $\Gamma^n$ in the numerical representation
of the kernel. This is clearly evident in the data shown in Table
\ref{Table:Diracs} for two choices of discretization level, $m=128,256$.
\begin{table}[tp]
\caption{Kernel based on Dirac delta functions supported along
  $\Gamma$.}
\label{Table:Diracs}
\begin{tabular}{|ccccc|ccccc|}
\hline
$m$&$n$&$\operatorname{cond}(M)$&$e^{m,n}_\infty$&$e^{m,n}_2$&
$m$&$n$&$\operatorname{cond}(M)$&$e^{m,n}_\infty$&$e^{m,n}_2$\\
\hline
128& 64& 10.9& 5.49e-02& 4.93e-02&256& 96& 14.34& 4.43e-02& 4.67e-02\\
& 80& 15.05& 4.06e-02& 2.72e-02&& 112& 18.06& 3.46e-02& 3.46e-02\\
& 96& 21.16& 2.59e-02& 1.57e-02&& 124& 21.49& 2.74e-02& 2.51e-02\\
& 112& 26.01& 1.25e-02& 7.50e-03&& 144& 25.58& 2.32e-02& 1.81e-02\\
& 128& 31.88& 7.06e-03& 2.31e-03&& 160& 29.83& 2.17e-02& 1.36e-02\\
& 144& 38.78& 1.05e-02& 7.72e-03&& 176& 36.52& 1.58e-02& 1.03e-02\\
256& 64& 8.42& 8.27e-02& 9.57e-02&& 192& 42.51& 1.20e-02& 7.80e-03\\
& 80& 11.26& 6.09e-02& 6.61e-02&& 208& 47.22& 1.13e-02& 5.66e-03\\\hline
\end{tabular}
\end{table}
Again the low condition number comes at the price of a
reduced accuracy (if the comparison is carried out at the same
discretization level $m$).
\subsubsection{Preconditioning}
Given the dramatic increase in condition number resulting from the use
of the proposed smoother kernels, it is natural to ask whether it can
be mitigated by some preconditioning procedure. Denote by
$M_\varphi$ and $M_\delta$ the matrix obtained discretizing the smooth
kernel  and the singular kernel, respectively, i.e.
$$
 M_\varphi=\big\langle \delta^m_{y_j},G^m_\pi\bigl(\varphi ^m_{\tilde
   y_k}\bigr)\big\rangle_{q^m} ,\: j,k=1,\dots,n, 
$$
and
$$
 M_\delta=\big\langle \delta^m_{y_j},G^m_\pi\bigl(\delta ^m_{y_k}
 \bigr)\big\rangle_{q^m} ,\: j,k=1,\dots,n.
$$
It seems natural to use the better conditioned but ``rough''
approximation $M_\delta$ as a preconditioner for the highly
accurate but badly conditioned $M_\varphi$. In Table
\ref{Table:Precon} the condition numbers of $M_\varphi$, $M_\delta$,
and $C=M_\delta^{-1}M_\varphi$ are shown for a few discretization
levels. They clearly point to an enormous benefit of preconditing. The
plots in Figure \ref{Figure:ContourM} gives a more visual
characterization of the effect of preconditioning on the diagonal
dominance of the corresponding matrix.
\begin{center}
\begin{figure}[!htb]
 \includegraphics[scale=0.3]{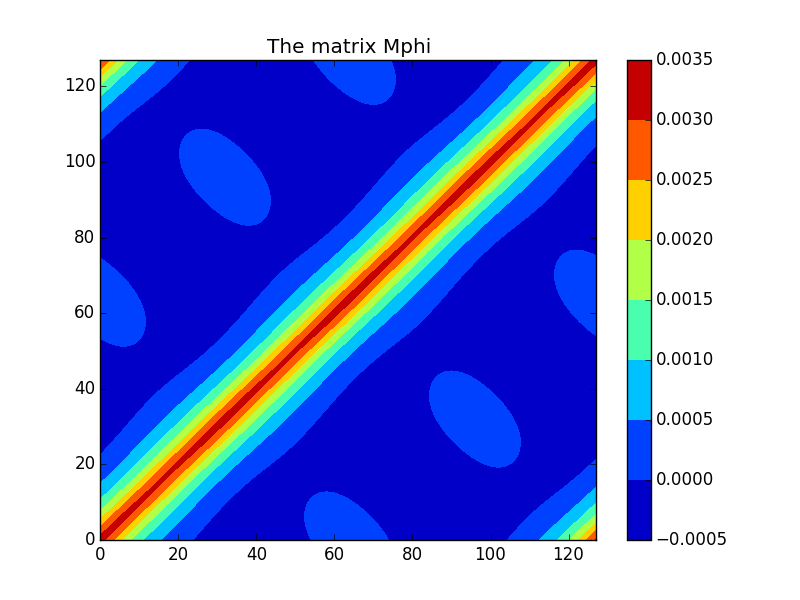}
\includegraphics[scale=0.3]{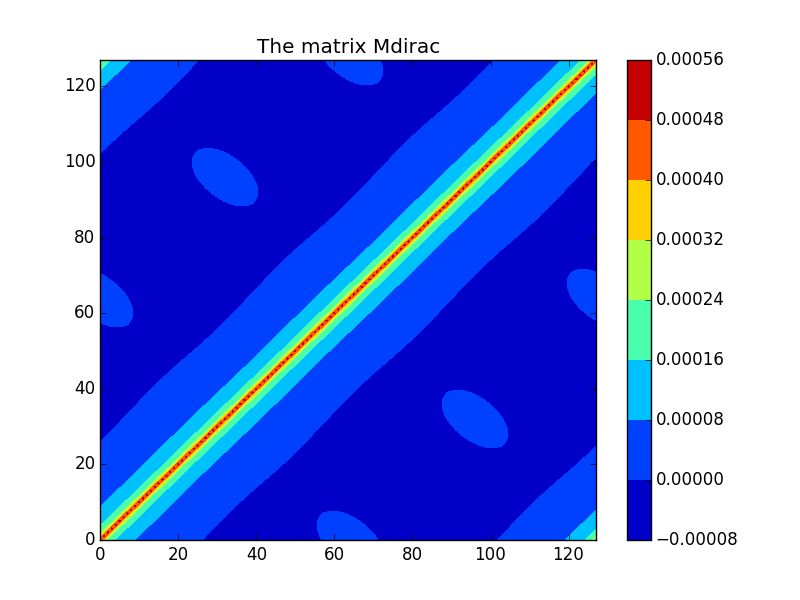}
\includegraphics[scale=0.3]{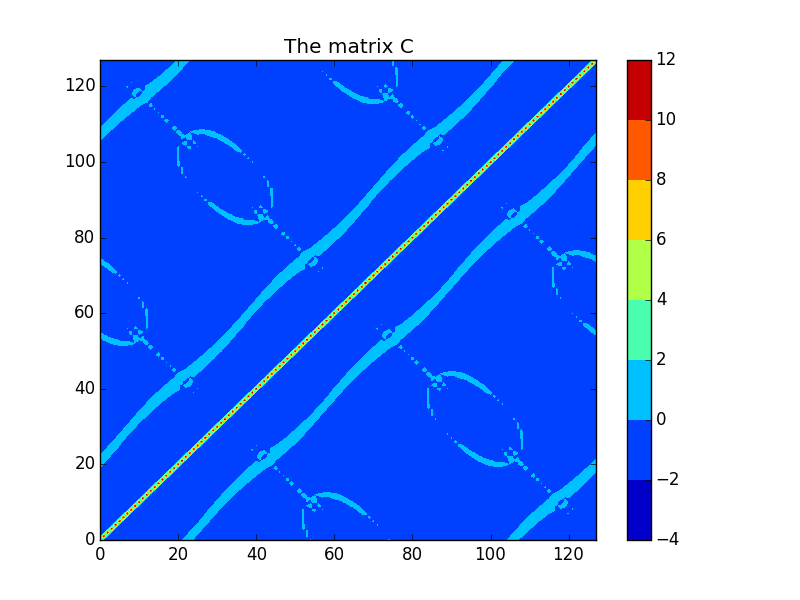}
 \caption{\label{Figure:ContourM} Contour plot of $M_\varphi$,
   $M_\delta$, and $C=M_\delta ^{-1}M_\varphi$ for $m=256$, $n=128$, and
   $\delta=0.4$.}
\end{figure}
\end{center}
\begin{table}[tp]
\caption{Preconditioning effect of $M_\delta^{-1}$ on $M_\varphi$ when
  $\delta=0.4$.} 
\label{Table:Precon}
\begin{tabular}{|cccc|}
\hline
 $(m,n)$&$\operatorname{cond}(M_\varphi)$&
 $\operatorname{cond}(M_\delta)$&
 $\operatorname{cond}(M_\delta^{-1}M_\varphi)$\\
\hline
 $(128,64)$&2.51e+03&1.09e+01&5.24e+00\\
 $(256,128)$&1.72e+06&2.15e+01&1.01e+01\\
$(512,256)$&4.02e+11&4.27e+01&1.96e+01\\
\hline
\end{tabular}
\end{table}
It can therefore be concluded that smoother kernels lead to higher
order resolutions and more accurate numerical results at the cost of
an apparent increase in condition number. Latter can, however, be completely
avoided by a simple and natural preconditioning procedure.
\subsubsection{Effective Numerical Implementation}\label{constcoeff}
The necessity to project a datum onto the subspace of mean zero
functions in the above procedure effectively destroys the translation
invariance of the constant coefficients equation on the periodic
box. This makes it necessary to compute a box solution for each entry
of the matrix $M$. While it was chosen to illustrate the ideas using
test-functions $\varphi _{\tilde{y}}$ approximating Dirac
distributions $\delta _y$ in order to harvest the benefits of the
theoretical analysis ensuring injectivity (and thus invertibility) of
$M$, it is clear that other 
choices are possible, such as normal derivatives of
testfunctions. These are particularly suited since 
they are mean zero functions supported in a small neighborhood of
their ``center-point''. As such they do not require to be projected
onto the mean free subspace. It is therefore enough to compute
$$
 (-\Delta _\pi)^{-1}\partial _{\nu(y)}\varphi _{\tilde y}=\sum
 _{j=1}^2\nu_j(y) (-\Delta _\pi)^{-1}\partial _j\varphi
 _{\tilde y} 
$$ 
for one point $y\in \Gamma$ only since
$$
 (-\Delta _\pi)^{-1}\partial _j\varphi _{\tilde y+v} =(-\Delta
 _\pi)^{-1}\tau _v (\partial _j\varphi _{\tilde y})=\tau _v (-\Delta
 _\pi)^{-1}\partial _j\varphi _{\tilde y},\: j=1,2,
$$
where $\tau _v u=u(\cdot -v)$ is the translation of a periodic
function $u$. This also gives insight into the ``circulant'' structure
of the matrix $M$.

It is also possible to replace the test-function centers $\{\tilde y\, :\, y\in
\Gamma\}$ by nearby or closest (box) grid points in $G^m$ so that the
translations required to obtain the kernel from the knowledge of, say, 
$(-\Delta _\pi)^{-1}\partial _{\nu(y_1)}\varphi _{\tilde y_1}$, can be
implemented efficiently (i.e. in physical space).
\begin{rem}
Notice that, if $\Delta$ is replaced by a more general
elliptic non-constant coefficient differential operator, the kernel
construction given above is still viable and would deliver a purely
numerical boundary integral method which does not rely on 
the explicit analytical knowledge of a fundamental solution for the
differential operator. It even allows replacing the ``discrete''
fundamental solution by a smooth kernel which can more accurately be
captured numerically. Remarkably this can be done at
effectively not cost due to the availability of the natural
preconditioning procedure  described above.
\end{rem}
\begin{rem}
The proposed construction of smooth kernels also suggests that
iterative parallelized methods can be used in the computation of the
entries of the matrix $M$ with a small number of iterations in the
case of a non-constant coefficient differential operator
$\mathcal{A}$, at least when the coefficients vary smoothly. This is
due to the fact that the building blocks $\mathcal{A_\pi}^{-1}\varphi
_{\tilde y}$ will be locally close to each other thus providing
excellent initial guesses for an iterative solver.
\end{rem}
\subsubsection{Kernel Functions}
Ultimately the accuracy of the method rests on its ability to
faithfully compute linearly independent functions in the kernel of 
the Laplacian $\Delta ^D_\Omega$ on the domain $\Omega$. These are
known explicitly for $\Omega=\mathbb{B}(0,2)$ and given by 
$$
 \psi_k(r,\theta)=(\frac{r}{2})^ke^{ik\theta},\: r\in[0,2],\:
 \theta\in[0,2\pi),\: k\in \mathbb{N},
$$
in polar coordinates. Using the method described above, it is possible
to compute a numerical approximation of these functions defined on
$G^m$. Tables \ref{Table:kfunctions1} and \ref{Table:kfunctions2} give
the relative errors observed for the first 33 kernel functions at two
distinct discretization levels.
\begin{center}
\begin{table}[!htb]
\caption{Resolution of the first 33 kernel functions for $m=128$,
  $n=80$, and $\delta=0.4$.}
\label{Table:kfunctions1}
\begin{tabular}{|c|cc||c|cc||c|cc|}
\hline
&$\ell _\infty$-err&$\ell _2$-err&&$\ell_\infty$-err&$\ell _2$-err  
&&$\ell _\infty$-err&$\ell _2$-err\\\hline
$\psi _{1}$&1.84e-07&8.11e-08&$\psi _{12}$&1.08e-05&9.60e-06&$\psi _{23}$&1.43e-03&9.88e-04\\
$\psi _{2}$&1.72e-07&9.16e-08&$\psi _{13}$&2.14e-05&1.39e-05&$\psi _{24}$&2.05e-03&1.57e-03\\
$\psi _{3}$&3.45e-07&2.08e-07&$\psi _{14}$&3.65e-05&2.28e-05&$\psi _{25}$&3.68e-03&2.36e-03\\
$\psi _{4}$&5.99e-07&3.60e-07&$\psi _{15}$&5.80e-05&3.13e-05&$\psi _{26}$&4.40e-03&3.51e-03\\
$\psi _{5}$&1.21e-06&5.73e-07&$\psi _{16}$&7.81e-05&5.81e-05&$\psi _{27}$&6.87e-03&5.60e-03\\
$\psi _{6}$&1.23e-06&8.65e-07&$\psi _{17}$&1.10e-04&7.28e-05&$\psi _{28}$&1.17e-02&9.09e-03\\
$\psi _{7}$&2.53e-06&1.32e-06&$\psi _{18}$&1.60e-04&1.04e-04&$\psi _{29}$&1.97e-02&1.32e-02\\
$\psi _{8}$&4.29e-06&2.66e-06&$\psi _{19}$&2.99e-04&1.72e-04&$\psi _{30}$&2.76e-02&2.14e-02\\
$\psi _{9}$&5.50e-06&2.92e-06&$\psi _{20}$&3.10e-04&2.63e-04&$\psi _{31}$&3.84e-02&3.08e-02\\
$\psi _{10}$&5.59e-06&4.38e-06&$\psi _{21}$&6.44e-04&4.12e-04&$\psi _{32}$&4.99e-02&4.63e-02\\
$\psi _{11}$&1.26e-05&6.32e-06&$\psi _{22}$&1.07e-03&6.35e-04&$\psi
                                                               _{33}$&9.99e-02&7.05e-02\\\hline
\end{tabular}
\end{table}
\end{center}
\begin{center}
\begin{table}[!htb]
\caption{Resolution of the first 33 kernel functions for $m=512$,
  $n=256$, and $\delta=0.4$.}
\label{Table:kfunctions2}
\begin{tabular}{|c|cc||c|cc||c|cc|}
\hline
&$\ell _\infty$-err&$\ell _2$-err&&$\ell_\infty$-err&$\ell _2$-err  
&&$\ell _\infty$-err&$\ell _2$-err\\\hline
$\psi _{1}$&1.24e-13&1.09e-14&$\psi _{12}$&1.57e-13&5.68e-14&$\psi _{23}$&6.23e-13&2.63e-13\\
$\psi _{2}$&9.17e-14&1.10e-14&$\psi _{13}$&1.90e-13&5.43e-14&$\psi _{24}$&7.15e-13&4.11e-13\\
$\psi _{3}$&1.29e-13&1.63e-14&$\psi _{14}$&1.34e-13&5.18e-14&$\psi _{25}$&9.98e-13&4.19e-13\\
$\psi _{4}$&1.23e-13&1.90e-14&$\psi _{15}$&2.03e-13&6.65e-14&$\psi _{26}$&1.25e-12&5.38e-13\\
$\psi _{5}$&1.22e-13&2.00e-14&$\psi _{16}$&2.73e-13&8.48e-14&$\psi _{27}$&1.52e-12&6.23e-13\\
$\psi _{6}$&1.04e-13&2.24e-14&$\psi _{17}$&2.65e-13&9.30e-14&$\psi _{28}$&1.48e-12&7.26e-13\\
$\psi _{7}$&1.31e-13&2.58e-14&$\psi _{18}$&3.68e-13&1.15e-13&$\psi _{29}$&1.99e-12&9.46e-13\\
$\psi _{8}$&1.14e-13&2.87e-14&$\psi _{19}$&3.73e-13&1.16e-13&$\psi _{30}$&1.80e-12&1.04e-12\\
$\psi _{9}$&1.23e-13&3.03e-14&$\psi _{20}$&4.03e-13&2.01e-13&$\psi _{31}$&2.77e-12&1.27e-12\\
$\psi _{10}$&1.10e-13&3.18e-14&$\psi _{21}$&5.31e-13&1.82e-13&$\psi _{32}$&3.38e-12&1.80e-12\\
$\psi _{11}$&1.74e-13&3.64e-14&$\psi _{22}$&4.84e-13&2.58e-13&$\psi _{33}$&3.50e-12&1.83e-12\\\hline
\end{tabular}
\end{table}
\end{center}
\subsection{Neumann Problem}
Next, using the same notations and discretization procedure, the
Neumann problem
\begin{equation}\label{neupb}
 \begin{cases}
 u-\triangle u=f&\text{in }\Omega=\mathbb{B}(0,2),\\
 \partial _\nu u=0&\text{on }\Gamma=\mathbb{S}^1_2,
 \end{cases}
\end{equation}
for $f(x)=\cos(\frac{\pi}{2}r)\bigl(1+\frac{\pi
  ^2}{4}\bigr)+\frac{\pi}{2}\sin(\frac{\pi}{2}r)/r$,
$r=\sqrt{x_1^2+x_2^2}$ and $x\in B$. This problem has the exact
solution $u$ given by $u(x)=\cos\bigl(\frac{\pi}{2}r(x)\bigr)$, $x\in
\Omega $. In order to show that the method is robust in the sense that
it does not depend on the exact choices of its ingredients, a
different cutoff function is used in order to modify the
right-hand-side $f$ to make it into a doubly periodic function which
fits the periodic framework. More specifically, take
$$
 \psi(x)=\frac{1}{2}\big\{ 1+\tanh \bigl(
 -\frac{5}{2}[r-(\pi-0.2)^2]\bigr) \big\},
$$
which essentially vanishes close to the the boundary of $B$ and takes
the value $1$ on $\Omega$. Replace then $f$ by $\tilde f=f\psi$ to obtain a
periodic function which coincides with $f$ on $\Omega$. In numerical
experiments, this is clearly performed on the grid, i.e. by replacing
$f^m$ by $\tilde f^m=\psi ^mf^m$. The solution procedure is
parallel to that employed for the Dirichlet problem. First the
function
$$
 v^m=(1^m-\triangle ^m)^{-1}\tilde f^m=\mathcal{F}^{-1}_m
 \operatorname{diag}\bigl( (\frac{1}{1+|k|^2})_{k\in
   \mathbb{Z}^2_m}\bigr)\mathcal{F}_m(\tilde f^m)
$$
is computed. Then the kernel matrix $M$ is obtain as
$$
 M_{jk}=\big\langle -\bigl(\partial
 _{\nu(y_j)}\delta_{y_j}\bigr)^m,(1^m-\triangle
 ^m)^{-1}\varphi^m_{\tilde y_k}\big\rangle _{q^m},\: 
 j,k=1,\dots,m,  
$$
where
$$
 -\bigl(\partial _{\nu(y_j)}\delta_{y_j}\bigr)^m=-\nu_1(y_j)(\delta
 _{y_j^1}')^m\otimes \delta^m_{y_j^2}-\nu_2(y_j)\delta^m
 _{y_j^1}\otimes(\delta '_{y_j^2})^m 
$$
is used as a discretization of the normal derivative operator at the
point $(y_j^1,y_j^2)=y_j\in \Gamma^n$. Finally the weight vector $w^n$
is determined by solving
$$
 \big \langle -\bigl(\partial _{\nu(y_j)}\delta_{y_j}\bigr)^m,v^m+\sum
 _{k=1}^nw^n_k(1^m-\triangle ^m)^{-1}\varphi ^m_{\tilde y_k}\big
 \rangle _{q^m}=z+Mw^n  =0,
$$
where $z_j=\langle -\bigl(\partial _{\nu(y_j)}\delta_{y_j}\bigr)^m,v^m
\rangle _{q^m}$ for $j=1,\dots,n$.
Results of similar numerical experiments to those performed for the
Dirichlet problem are summarized in Table \ref{Table:Neumann}.
\begin{rem}
Notice that in all numerical experiments, radially symmetric functions
were used. One reason is that radial symmetry is not readily
compatible with periodicity in that it cannot be represented with very
few periodic modes. Another is that explicit formul{\ae} are
available. 
\end{rem}
\begin{table}[tp]
\caption{Numerical experiments for the Neumann problem \eqref{neupb}.}
\label{Table:Neumann}
\begin{tabular}{|cccccc|cccccc|}\hline
$m$&$n$&$\delta$&$\operatorname{cond}(M)$&$e^{m,n}_\infty$&$e^{m,n}_2$
&$m$&$n$&$\delta$&$\operatorname{cond}(M)$&$e^{m,n}_\infty$&$e^{m,n}_2$
\\\hline
32& 32& 0.3& 3.31e+00& 4.84e-03& 2.97e-03&128& 64& 0.3& 2.82e+01& 2.59e-04& 2.55e-04\\
& & 0.4& 6.04e+00&8.67e-03& 4.94e-03&& & 0.4& 9.83e+01& 2.12e-05& 2.02e-05\\
& & 0.5& 9.61e+00& 3.99e-03& 3.86e-03&& & 0.5& 3.27e+02& 6.24e-06& 4.44e-06\\
& 48& 0.3& 9.98e+00& 1.31e-02& 1.20e-02&256& 64& 0.3& 2.84e+01& 2.65e-04& 2.56e-04\\
& & 0.4& 2.85e+01& 1.22e-02& 9.84e-03&& & 0.4& 9.82e+01& 1.91e-05& 1.84e-05\\
& & 0.5& 1.79e+02& 3.37e-03& 2.35e-03&& & 0.5& 3.26e+02& 1.53e-06& 1.47e-06\\
64& 32& 0.3& 3.14e+00& 2.18e-02& 2.08e-02&& 128& 0.3& 2.46e+03& 3.41e-08& 3.34e-08\\
& & 0.4& 5.45e+00& 5.58e-03& 5.57e-03&& & 0.4& 3.33e+04& 4.41e-10& 1.75e-10\\
& & 0.5& 9.38e+00& 2.39e-03& 2.19e-03&& & 0.5& 4.08e+05& 5.96e-10& 3.87e-10\\
& 48& 0.3& 9.36e+00& 1.70e-03& 1.36e-03&512& 128& 0.3& 2.46e+03& 3.43e-08& 3.35e-08\\
& & 0.4& 2.33e+01& 5.51e-04& 3.24e-04&& & 0.4& 3.33e+04& 2.65e-10& 1.93e-10\\
& & 0.5& 5.35e+01& 2.94e-04& 1.42e-04&& & 0.5& 4.08e+05& 1.05e-10& 3.95e-11\\
& 64& 0.3& 2.90e+01& 1.38e-04& 5.22e-05&& 256& 0.3& 1.88e+07& 1.76e-10& 7.73e-11\\
& & 0.4& 9.89e+01& 6.46e-04& 5.20e-04&& & 0.4& 3.88e+09& 1.76e-10& 7.72e-11\\
& & 0.5& 3.48e+02& 3.98e-04& 3.53e-04&&&0.5& 6.47e+11& 1.76e-10&7.72e-11\\\hline
\end{tabular}
\end{table}
\begin{rem}
While it might appear that in the construction of the matrix kernel
$M$, one needs to solve $n$ problems in the discretized periodicity 
box, this is not always the case. As for the Dirichlet problem, the
operator $1-\triangle$ is translation invariant. It follows that it is
enough to solve one such problem, e.g. for $k=1$ since all other
solutions would be a translate of the solution for $k=1$. This is true
because the datum $\varphi _{\tilde y_k}$ is a translate of
$\varphi_{\tilde y_1}$. To make sure that the translation be
compatible with the grid $G^m$, the theoretical location $\tilde 
y=y+\delta \nu_\Gamma(y)$ would have to be replaced by the closest
grid point in $G^m$ (for instance).
\end{rem}
\section{Conclusions}
An effectively meshless approach to boundary value problems in general
geometry domains is proposed based on the use of uniform
discretizations of an encopassing computational box. Exploiting a
pseudodifferential operator framework, relevant kernels can be
replaced by smoother kernels which allow for more accurate numerical
resolution. No explicit knowledge of the kernels is required beyond
their analytical structure which is used in an essential way in order
to construct their numerical counterparts. While the smooth kernels,
which correspond to infinitely smoothing compact operators, and their
associated discretization matrices are badly ill-conditioned, they can
very effectively be preconditioned by use of their ``rougher''
counterparts with singular kernels in an arguably natural way at
minimal additional cost. The methodology proposed is very general and
can be employed in three space dimensions as well as to more general
linear and nonlinear boundary value problems. The fact that no
remeshing is required makes this method particularly appealing for
free and moving boundary problems. These extensions will be the topic of
forthcoming papers. 
\bibliography{lite.bib}
\end{document}